\documentclass{amsart}
\usepackage{graphicx,dsfont}
\usepackage[dvipdfm, colorlinks, linkcolor=blue, anchorcolor=green, citecolor=red]{hyperref}
\usepackage{mathpazo}
\usepackage[all]{xy}
\newtheorem{theorem}{Theorem}[section]
\newtheorem{corollary}[theorem]{Corollary}
\newtheorem{lemma}[theorem]{Lemma}
\newtheorem{proposition}[theorem]{Proposition}
\theoremstyle{definition}
\newtheorem{definition}[theorem]{Definition}
\newtheorem{remark}[theorem]{Remark}
\newtheorem{example}[theorem]{Example}
\numberwithin{equation}{section}
\begin{document}
\title[The chord index, its definitions, applications and generalizations]{The chord index, its definitions, applications and generalizations}%
\author{Zhiyun Cheng}%
\address{School of Mathematical Sciences, Laboratory of Mathematics and Complex Systems, Beijing Normal University, Beijing 100875, China}%
\email{czy@bnu.edu.cn}%
\begin{abstract}
In this paper we study the chord index of virtual knots, which can be thought of as an extension of the chord parity. We show how to use the chord index to define finite type invariants of virtual knots. The notions of indexed Jones polynomial and indexed quandle are introduced, which generalize the classical Jones polynomial and knot quandle respectively. Some applications of these new invariants are discussed. We also study how to define a generalized chord index via a fixed finite biquandle. Finally the chord index and its applications in twisted knot theory are discussed.
\end{abstract}
\maketitle
\section{Introduction}
This paper concerns with the chord index and it applications in virtual knot theory and twisted knot theory. Virtual knot theory, which was introduced by L. Kauffman in \cite{Kau1999}, studies the embeddings of $S^1$ in $\Sigma_g\times [0, 1]$ up to isotopy and stabilizations. Here $\Sigma_g$ denotes a closed orientable surface with genus $g$. When $g=0$, virtual knot theory reduces to the classical knot theory. It was first observed by Kauffman \cite{Kau2004} that each real crossing point of a virtual knot can be assigned with a parity, and a kind of self-linking number, the odd writhe, was proved to be a virtual knot invariant. Later this idea was extended by Mantutov in \cite{Man2010}. In \cite{Che2013} we introduced the notion of chord index, which assigns an integer to each real crossing point such that the parity of it exactly equals the parity that introduced by Kauffman. The main aim of this paper is to provide some applications of the chord index in virtual knot theory and its extension, the twisted knot theory.

First, we would like to discuss how to construct finite type invariants of virtual knots by using chord index. A well-known result in finite type invariant theory is, for classical knots there is no finite type invariant of degree one. However for virtual knots this is not the case. In \cite{Saw2003} Sawollek used a degree one finite type invariant to distinguish between a virtual knot and its inverse. Later in \cite{Hen2010} Henrich defined three degree one finite type invariants for virtual knots, and the strongest one is a ``universal" degree one finite type invariant for virtual knots (do not confuse this ``universal" invariant with the Kontsevich integral \cite{Kon1993}, we refer the reader to \cite{Hen2010} for the precise definition of this ``universal" finite type invariant). The first application of the chord index is investigating how to construct finite type invariants of virtual knots in general cases.

As the second application of the chord index, we introduce a sequence of Jones polynomials, say the indexed Jones polynomial. By ignoring all virtual crossing points, the classical Jones polynomial can be naturally defined for virtual knots with the help of Kauffman bracket. We remark that if $K$ is a classical knot, then Jones polynomial $V_K(t)$ takes value in $\mathds{Z}[t^{\pm1}]$. However for a virtual knot $K$, the Jones polynomial of $K$ takes value in $\mathds{Z}[t^{\pm\frac{1}{2}}]$. Therefore if $V_K(t)$ contains nonzero coefficient for some term $t^{\frac{n}{2}}$ $(n\neq0)$, then we conclude that $K$ is not classical. On the other hand when $K$ is a proper alternating virtual knot diagram, N. Kamada proved that the span of $V_K(t)=c(K)-g(K)$ \cite{Kam2004}, here $c(K)$ and $g(K)$ denote the crossing number and supporting genus of $K$ respectively. Later in \cite{Man2010} the classical Jones polynomial was generalized by Manturov to the parity skein relation polynomial invariant. Similar idea was used to define the parity arrow polynomial and its categorification \cite{Kae2012}. In this paper, by using the chord index, the set of real crossing points is divided into several subsets. For each subset we introduce an indexed Jones polynomial associated to it. Analogous to the classical case, we show that each indexed Jones polynomial provides a lower bound for the cardinality of the corresponding subset.

Thirdly, we introduce the notion of indexed quandle. Roughly speaking, an indexed quandle is a set with a sequence of binary operations (indexed by $\mathds{Z}$) which satisfies certain axioms. When all operations coincide the indexed quandle reduces to the classical quandle structure, which was first introduced in \cite{Joy1982, Mat1984}. Therefore the indexed quandle can be thought of as an extension of the classical quandle. For each virtual knot $K$ we define the indexed knot quandle of it, denoted by $IndQ(K)$. This invariant is equivalent to the fundamental quandle of $K$ when $K$ is a classical knot. But for virtual knots, we give some examples to show that it contains much more information than the fundamental quandle. In particular, with a given finite indexed quandle $Q$ one can define the coloring invariant $Col_Q(K)$ by counting the homomorphisms from $IndQ(K)$ to $Q$. As an analogue of the quandle cocycle invariants \cite{Car2003}, we define the indexed quandle cocycle invariants. Some examples are given to reveal that this cocycle invariant is more powerful than the coloring invariant.

In Section 6 the definition of the chord index is revisited. We want to understand what is a chord index essentially. As we understand it, the chord index can be regarded as a particular biquandle cocycle. In this way we discuss how to generalize the definition of the chord index for virtual links.

The last section is devoted to investigate the chord index and its applications in twisted knot theory.

\section{Virtual knot theory and chord index}
\subsection{A brief review of virtual knots}
Let $\Sigma_g$ be a closed orientable surface with genus $g$ and $K$ an embedded circle in $\Sigma_g\times[0, 1]$. Assume we have another embedded circle $K'\subset \Sigma_{g'}\times[0, 1]$, we say $K$ and $K'$ are \emph{stably equivalent} if one can be obtained from the other one by isotopy in the thickened surfaces, homeomorphisms of the surfaces and addition or subtraction of empty handles. We define the \emph{virtual knots} to be the stable equivalence classes of circles embedded in thickened surfaces. For a virtual knot $K$ the minimal genus of the surface $\Sigma_g$ is called the \emph{supporting genus} of $K$. By using some classical technique in 3-manifold topology, Kuperberg \cite{Kup2003} proved that the embedding of a virtual knot in the minimal supporting genus thickened surface is unique. It follows that if two classical knots are stably equivalent as virtual knots, then they are also equivalent as classical knots. This implies the virtual knot theory is indeed an extension of the classical knot theory.

From the diagrammatic viewpoint a virtual knot can be interpreted by virtual knot diagrams. A virtual knot diagram is an immersed circle in the plane with finitely many double points. By replacing each double point with an overcrossing, or an undercrossing, or a virtual crossing (denoted by a small cricle) we obtain a virtual knot diagram. Obviously if there exists no virtual crossing the virtual knot diagram represents a classical knot. We say a pair of virtual knot diagrams are \emph{equivalent} if they can be connected by a sequence of generalized Reidemeister moves, see Figure \ref{figure1}.
\begin{figure}[h]
\centering
\includegraphics{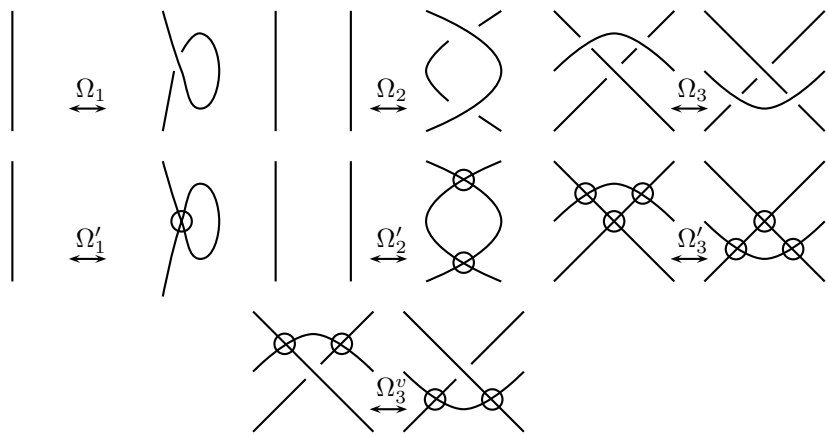}\\
\caption{Generalized Reidemeister moves}\label{figure1}
\end{figure}
Now we can define \emph{virtual knots} as the equivalence classes of virtual knot diagrams up to generalized Reidemeister moves.

The two definitions above are closely related. Assume we have a embedded circle in a thickened surface, now consider a projection of the circle to the plane in general position. If the preimage of a double point is an overcrossing (undercrossing) in the thickened surface, then we still use an overcrossing (undercrossing) to denote it. If the two strands of the preimage of a double point locate in two different levels, then we use a virtual crossing to denote it. In other words, the virtual crossings can be regarded as artifacts of the projection of the surface to the plane. The readers are referred to \cite{Kau1999} for more details. Conversely, suppose we have a virtual knot diagram $K$ on the plane. By taking one-point compactification we obtain a virtual knot diagram on $S^2$. For each virtual crossing we add a 2-handle locally to eliminate the crossing. Finally we obtain an embedded circle in $\Sigma_{c_v(K)}\times[0, 1]$, where $c_v(K)$ denotes the number of virtual crossings in $K$. The following theorem shows that the two definitions above are equivalent.
\begin{theorem}[\cite{Kau1999,Car2002}]
Two virtual knot diagrams are equivalent if and only if their corresponding surface embeddings are stably equivalent.
\end{theorem}

Another way to understand virtual knots is to regard them as Gauss diagrams. Let $K$ be a virtual knot diagram, which can be seen as an immersed circle in the plane. Consider the preimage of this immersed circle with an anticlockwise orientation. For each real crossing point, we draw a chord directed from the preimage of the overcrossing to the preimage of the undercrossing. Finally we assign a sign to each chord according to the sign(writhe) of the corresponding crossing point. We call this chord diagram the \emph{Gauss diagram} of $K$ and use $G(K)$ to denote it, see Figure \ref{figure2} for a simple example. We note that all virtual crossing points are ignored on the Gauss diagram.
\begin{figure}[h]
\centering
\includegraphics{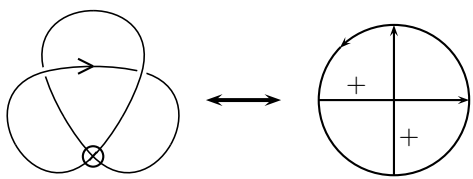}\\
\caption{Virtual trefoil knot and its Gauss diagram}\label{figure2}
\end{figure}

It is well known that each classical knot diagram has a corresponding Gauss diagram, but the converse is not true in general. If we use virtual knot diagrams instead of the classical knot diagrams, then for each Gauss diagram we can always find a corresponding virtual knot diagram that represents it. Although there may exist infinitely many different virtual knot diagrams which correspond to the same Gauss diagram, we have the following correspondence between them.
\begin{theorem}[\cite{Gou2000}]
A Gauss diagram uniquely defines a virtual knot isotopy class.
\end{theorem}

Since the time when virtual knot theory was introduced, many virtual knot invariants have been introduced. Several classical knot invariants can be directly extended to the virtual world. For example, the knot group, the knot quandle and the Jones polynomial can be similarly defined for virtual knots \cite{Kau1999}. Some generalizations of the Alexander polynomial for virtual knots can be found in \cite{Saw2001} and \cite{Sil2003}, and some generalizations of the Jones polynomial can be found in \cite{Miy2006,Miy2008} and \cite{Dye2009}. Readers should refer to \cite{Fen2014} for some recent progress and open problems in virtual knot theory.

\subsection{Chord index}
In the present paper we will focus on the virtual knot invariants induced from the chord index. Roughly speaking, a chord index is an integer assigned to each chord in a Gauss diagram, or equivalently to each real crossing point of the virtual knot diagram. We are going to give two definitions of the chord index, one comes from Gauss diagrams and the other one comes from knot diagrams.

Let $K$ be a virtual knot diagram and $G(K)$ the corresponding Gauss diagram. According to the one to one correspondence between the real crossing points in $K$ and chords in $G(K)$, we will use the same notation to denote a real crossing in $K$ and its corresponding chord in $G(K)$. Choose a chord $c$ in $G(K)$, we associate four integers to $c$ as follows:
\begin{enumerate}
  \item $r_+(c)=$ the number of positive chords crossing $c$ from left to right;
  \item $r_-(c)=$ the number of negative chords crossing $c$ from left to right;
  \item $l_+(c)=$ the number of positive chords crossing $c$ from right to left;
  \item $l_-(c)=$ the number of negative chords crossing $c$ from right to left.
\end{enumerate}
\begin{figure}[h]
\centering
\includegraphics[width=3cm]{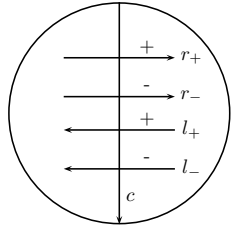}\\
\caption{The definition of the chord index}\label{figure3}
\end{figure}
Now we define the \emph{index} of $c$ as
\begin{center}
Ind$(c)=r_+(c)-r_-(c)-l_+(c)+l_-(c)$.
\end{center}
Roughly speaking, the index of a chord $c$ counts the signed sum of the chords which have nonempty intersections with $c$. In other words, each chord that has nonempty intersection with $c$ contributes $\pm1$ to Ind$(c)$.

Now we follow \cite{Fol2013} to give another definition of the chord index from the viewpoint of knot diagrams. Before proceeding to give the definition, we need to take a quick review of the linking number in virtual knot theory. Let $L=K_1\cup K_2$ be a 2-component virtual link diagram. We use Over$(C)$ (Under$(C)$) to denote the set of crossings between $K_1$ and $K_2$ that we encounter as overcrossings (undercrossings) when we travel along $K_1$. Now we define the \emph{over linking number} $lk_O(L)=\sum\limits_{c\in\text{Over}(C)}w(c)$ and the \emph{under linking number} $lk_U(L)=\sum\limits_{c\in\text{Under}(C)}w(c)$, where $w(c)$ is the writhe of $c$. Note that if $L$ is classical, we always have $lk_O(L)=lk_U(L)$. But when $L$ has some virtual crossings, this is not true in general.

We turn to the definition of chord index using over linking number and under linking number. Let $K$ be a virtual knot diagram and $c$ a real crossing point of it. By smoothing $c$ along the orientation of $K$ we obtain a 2-component link $L=K_1\cup K_2$, where the order of $K_1$ and $K_2$ is indicated in Figure \ref{figure4}. The \emph{index} of the crossing point $c$ can be defined as below
\begin{center}
Ind$(c)=lk_O(L)-lk_U(L)$.
\end{center}
\begin{figure}[h]
\centering
\includegraphics{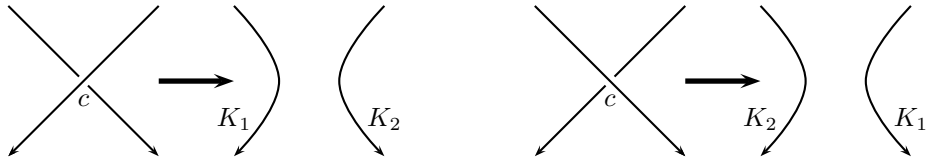}\\
\caption{Smooth the crossing point $c$}\label{figure4}
\end{figure}

We end this section with some useful properties of the chord index. The details of the proof can be found in, for example \cite{Che2013,Fol2013} and \cite{Kau2013}.
\begin{proposition}
Let $K$ be a virtual knot diagram and $c$ a real crossing point of it, then we have the following results:
\begin{enumerate}
  \item The two definitions of the chord index mentioned above are equivalent.
  \item If $c$ is isolated, i.e. no other chord has nonempty intersection with $c$, then Ind$(c)=0$.
  \item The two crossings involved in $\Omega_2$ have the same index.
  \item The indices of the three crossings involved in $\Omega_3$ are invariant under $\Omega_3$.
  \item $\Omega_i$ $(i=1, 2, 3)$ preserves the indices of chords that do not appear in $\Omega_i$ $(i=1, 2, 3)$.
  \item If $K$ contains no virtual crossings, then every crossing of $K$ has index zero.
  \item Ind$(c)$ is invariant under switching some other real crossings.
\end{enumerate}
\end{proposition}

\section{finite type invariants of virtual knots}
\subsection{Finite type invariants}
A finite type invariant (Vassiliev invariant) is a knot invariant which takes values in an abelian group, and it vanishes on all singular knots with $n$ singularities if $n$ is greater than some fixed integer. Finite type invariant was first introduced by Vassiliev in \cite{Vas1990} and later reformulated by Birman and Lin in a combinatorial way \cite{Bir1993}. As Kauffman did in \cite{Kau1999}, the definition of the finite type invariant can be directly extended to virtual knots. When there is no virtual crossing point, this definition coincides with the combinatorial definition given in \cite{Bir1993}.

Before defining the finite type invariants we need to take a quick review of the singular virtual knot theory. By a singular virtual link diagram, we mean a 4-valent planar graph with some vertices replaced by real crossings and some vertices replaced by virtual crossings. For the remaining crossings, we call them \emph{singular crossings}. Two singular virtual link diagrams are \emph{equivalent} if and only if one can be obtained from the other one by a sequence of generalized Reidemeister moves (Figure \ref{figure1}) and singular Reidemeister moves (Figure \ref{figure5}).
\begin{figure}[h]
\centering
\includegraphics{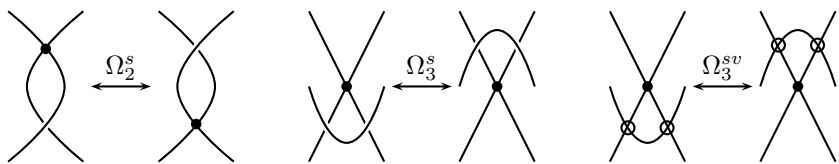}\\
\caption{Singular Reidemeister moves}\label{figure5}
\end{figure}

Let $f$ be a virtual knot invariant which take values in an abelian group. We extend $f$ to an invariant of singular virtual knots with $n$ singularities via the following recursive relation
\begin{center}
$f^{(n)}(K)=f^{(n-1)}(K_+)-f^{(n-1)}(K_-)$,
\end{center}
here $K_+$ is obtained from $K$ by resolving a singular crossing point into a positive crossing and $K_-$ is obtained from $K$ by resolving the same singular crossing point into a negative crossing. For the initial condition we set $f^{(0)}=f$. We say a virtual knot invariant $f$ is a \emph{finite type invariant of degree $n$} if $f^{(n+1)}$ vanishes on all singular virtual knots with $n+1$ singularities, but there exists a singular virtual knot $K$ with $n$ singularities such that $f^{(n)}(K)\neq0$. In other words, if $f$ is a finite type invariant of degree $n$, then for any singular virtual knot diagram $K$ with $n+1$ singular crossings we have
\begin{center}
$\sum\limits_{\sigma\in\{0, 1\}^{n+1}}(-1)^{|\sigma|}f(K_{\sigma})=0$.
\end{center}
Here $\sigma$ runs over all $(n+1)$-tuples of zeros and ones, $|\sigma|$ denotes the number of ones in $\sigma$ and $K_{\sigma}$ is obtained from $K$ by replacing the $i$-th singular crossing with a positive (negative) crossing if the $i$-th position of $\sigma$ is zero (one).

For example, for classical knots, a finite type invariant of degree 0 means it is preserved under crossing change, hence it takes the same value for all knots. It is easy to observe that there is no finite type invariant of degree 1 for classical knots. Later we will find that this is not the case for virtual knots.

\subsection{Writhe polynomial}
As before we use $K$ to denote a virtual knot diagram and $c$ to denote a real crossing of $K$. Proposition 2.3 enlightens us to consider the following integers
\begin{center}
$a_n(K)=
\begin{cases}
\sum\limits_{\text{Ind}(c)=n}w(c)& \text{if }n\neq0;\\
\sum\limits_{\text{Ind}(c)=n}w(c)-w(K)& \text{if }n=0,
\end{cases}$
\end{center}
where $w(c)$ and $w(K)$ denote the writhe of $c$ and $K$ respectively. We note that $a_0$ in fact can be determined by other $a_n$ $(n\neq 0)$. More precisely,
\begin{center}
$a_0(K)=\sum\limits_{\text{Ind}(c)=0}w(c)-w(K)=\sum\limits_{\text{Ind}(c)=0}w(c)-\sum\limits_{c}w(c)=-\sum\limits_{\text{Ind}(c)\neq 0}w(c)=-\sum\limits_{n\neq 0}a_n(K)$.
\end{center}
The following theorem can be easily derived from Proposition 2.3.
\begin{theorem}[\cite{Che2013,Dye2013,Im2013,Kau2013,Shi2014}]
For each $n\in \mathds{Z}$, $a_n(K)$ is a virtual knot invariant.
\end{theorem}
For convenience, we rewrite these invariants $\{a_n(K)\}$ in the form of a polynomial. We define the \emph{writhe polynomial}, which was introduced in \cite{Che2013}, as
\begin{center}
$W_K(t)=\sum\limits_{n\neq0}a_n(K)t^n$.
\end{center}
Obviously $W_K(t)$ is also a virtual knot invariant. If we want to include the contributions from the real crossings with index zero, following \cite{Kau2013} we define the \emph{affine index polynomial} $P_K(t)$ to be
\begin{center}
$\sum\limits_na_n(K)t^n=W_K(t)+a_0(K)=W_K(t)-W_K(1)$.
\end{center}

The writhe polynomial $W_K(t)$ has many applications in virtual knot theory. First, $W_K(t)=0$ if $K$ is classical, since all crossing points have index zero in this case. Hence whenever $W_K(t)\neq0$, then $K$ must be non-classical. On the other hand, the writhe polynomial is quite sensitive to some symmetries of virtual knots. For example, consider the virtual trefoil knot in Figure \ref{figure2}. Direct calculation shows that the writhe polynomial of it is $t+t^{-1}$, however the mirror image of it has writhe polynomial $-t-t^{-1}$. Some examples in \cite{Che2013} show that writhe polynomial also can be used to distinguish some virtual knots from their inverses. According to the definition of the chord index, it is evident that $|a_n(K)|$ $(n\neq0)$ gives a lower bound for the number of crossings with index $n$. However, if a crossing point has index zero, then it has no contribution to the writhe polynomial. Therefore in general we can not obtain any information about the number of crossing points with index zero from the writhe polynomial. For example, the index of any crossing in a classical knot diagram is zero but the writhe polynomial is also zero. One approach to overcome this problem was given in \cite{Che2016} recently. In Section 4 we will give an alternative solution to this problem.

\subsection{Finite type invariants of degree 0}
In order to discuss finite type invariants of degree 0, it is convenient to consider the flat virtual knot theory. A \emph{flat virtual knot diagram} is an immersed $S^1$ in the plane, where each crossing is either a virtual crossing point or a flat crossing point. Two flat virtual knot diagrams are \emph{equivalent} if they are related by finitely many flat Reidemeister moves indicated in Figure \ref{figure1}, where all real crossings should be replaced by flat crossings. Roughly speaking, flat virtual knots are equivalence classes of virtual knots up to crossing changes. Therefore, if a flat virtual knot diagram has no virtual crossing point, then it must be trivial. Given a virtual knot diagram $K$, one can define the corresponding flat virtual knot $F(K)$ by replace all real crossings in $K$ with flat crossings. Notice that if $K$ and $K'$ can be connected by a sequence of generalized Reidemeister moves, then $F(K)$ and $F(K')$ can be connected by a sequence of corresponding flat Reidemeister moves. This makes a guarantee that the definition of $F(K)$ is well defined.
\begin{remark}
We remark that each flat virtual knot diagram obviously has a corresponding Gauss diagram, where each chord in the Gauss diagram has no direction or sign. However for a given Gauss diagram, in general it corresponds to infinitely many different flat virtual knot diagrams. All these flat virtual knots are equivalent if we add one more move in the equivalence relations. This leads to the \emph{free knot theory} introduced by Manturov. The readers are referred to \cite{Man2010} and references therein for more details.
\end{remark}

Recall that a finite type invariant of degree 0 is a nonzero virtual knot invariant which vanishes on all singular virtual knots with one singularity. According to the recursive relation, it follows that a nonzero virtual knot invariant is a finite type invariant of degree 0 if and only if it is invariant with respect to crossing changes. Therefore if we have a nonzero flat virtual knot invariant $g$ (for example, the counting invariant derived from semiquandle introduced in \cite{HN2010}), then by defining $f(K)=g(F(K))$ we obtain a finite type invariant of degree 0. Conversely, each finite type invariant of degree 0 also provides us a flat virtual knot invariant.

Now we define a flat virtual knot invariant (hence a finite type invariant of degree 0) based on the writhe polynomial $W_K(t)$. Let $K$ be a virtual knot diagram and $F(K)$ the corresponding flat virtual knot. We define a polynomial of $F(K)$ as below
\begin{center}
$\mathfrak{F}_{F(K)}(t)=W_K(t)-W_K(t^{-1})$.
\end{center}
\begin{proposition}
$\mathfrak{F}_{F(K)}(t)$ is a well defined flat virtual knot invariant.
\end{proposition}
\begin{proof}
It suffices to prove that if another virtual knot $K'$ can be obtained from $K$ by some generalized Reidemeister moves and crossing changes, then $\mathfrak{F}_{F(K)}(t)=\mathfrak{F}_{F(K')}(t)$. Since $W_K(t)$ is invariant under generalized Reidemeister moves, it is sufficient to consider the case that $K'$ can be obtained from $K$ by taking crossing change on a crossing point $c$. Without loss of generality, we assume the index of $c$ in $K$ equals to $k$ and the writhe of it is positive. Then after switching $c$ this crossing has index $-k$ and negative sign in $K'$. Note that the indices and writhes of all other crossings are preserved. Now we have
\begin{center}
$\mathfrak{F}_{F(K)}(t)=f(t)+t^k-f(t^{-1})-t^{-k}=f(t)-t^{-k}-f(t^{-1})+t^{k}=\mathfrak{F}_{F(K')}(t)$.
\end{center}
The proof is finished.
\end{proof}
\begin{example}
Consider the virtual knot $K$ and its corresponding flat virtual knot $F(K)$ in Figure \ref{figure6}. Direct calculation shows that $W_K(t)=t^2+2t^{-1}$. Hence $\mathfrak{F}_{F(K)}(t)=t^2-2t+2t^{-1}-t^{-2}$, which means $F(K)$ is a nontrivial flat virtual knot.
\end{example}
\begin{figure}[h]
\centering
\includegraphics{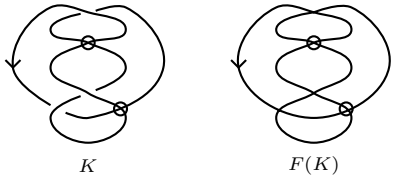}\\
\caption{Virtual knot $K$ and its corresponding flat virtual knot $F(K)$}\label{figure6}
\end{figure}
\begin{corollary}
Let $K$ be a virtual knot, if $W_K(t)\neq W_K(t^{-1})$, then the corresponding flat virtual knot $F(K)$ is nontrivial.
\end{corollary}

\subsection{Finite type invariants of degree 1}
As we mentioned in the introduction, several finite type invariants of degree 1 have been discussed in \cite{Saw2003} and \cite{Hen2010}. For the coefficients of the affine index polynomial we have the following result.
\begin{proposition}[\cite{Dye2013}]
For each virtual knot $K$ and any $n\in\mathds{Z}$, $a_n(K)$ is a finite type invariant of degree 1.
\end{proposition}
\begin{proof}
First we prove that $a_n^{(2)}(K)$ vanishes on each singular virtual knot with 2 singularities. Let $K$ be a virtual knot diagram and $c_1, c_2$ two real crossing points of $K$. Without loss of generality we assume $w(c_1)=w(c_2)=+1$, Ind$(c_1)=a$ and Ind$(c_2)=b$. Denote the diagram obtained from $K$ by switching $c_1$ as $K_{-+}$. Similarly we can define $K_{++}, K_{+-}$ and $K_{--}$. Note that $K_{++}=K$. Then we have
\begin{center}
$P_{K_{++}}(t)=f(t)+t^a+t^b$,\\ $P_{K_{-+}}(t)=f(t)-t^{-a}+t^b+2$,\\
$P_{K_{+-}}(t)=f(t)+t^{a}-t^{-b}+2$,\\ $P_{K_{--}}(t)=f(t)-t^{-a}-t^{-b}+4$.
\end{center}
It follows that $P_{K_{++}}(t)-P_{K_{-+}}(t)-P_{K_{+-}}(t)+P_{K_{--}}(t)=0$.

Now for each $n\in \mathds{Z}$ it suffices to find a singular virtual knot with one singularity such that $a_n^{(1)}$ is nonzero. Consider a Gauss diagram which consists of $n\geq1$ parallel positive chords (with the same direction) and one other positive chord $c$ crossing all of them, such that Ind$(c)=n$. Replace the real crossing $c$ with a singular crossing we obtain a singular virtual knot $K$. According to the definition we have
\begin{center}
$P_K^{(1)}(t)=(t^n+nt^{-1}-n-1)-(-t^{-n}+nt^{-1}-n+1)=t^n-t^{-n}-2$,
\end{center}
which implies $a_n^{(1)}(K), a_{-n}^{(1)}(K)$ and $a_0^{(1)}(K)$ are nonzero.
\end{proof}

There are several different kinds of generalization of the writhe polynomial. For example, Henrich in \cite{Hen2010} defined three virtual knot invariants. The first one is special case of the writhe polynomial. The other two invariants, called the smoothing invariant and the gluing invariant, are polynomials of flat virtual knots and singular flat virtual knots respectively. Henrich proved they are all finite type invariant of degree one. One can simply generalize the writhe polynomial  using the idea of smoothing invariant (see also \cite{Dye2013}) and gluing invariant. In \cite{Che2016} we gave a completely different extension of the writhe polynomial by replacing the chord index with an index function. It was proved that this generalized invariant is also a finite type invariant of degree one. We refer the readers to \cite{Che2016} for more details.

\subsection{Finite type invariants of higher degrees}
In \cite{Pol1994}, Michael Polyak and Oleg Viro gave a description of finite type invariants of degree two and three by means of Gauss diagram representations. By using virtual knots, Mikhail Goussarov \cite{Gou2000} proved that any finite type invariant of classical knots can be described by a Gauss diagram formula. In other words, any finite type invariant of classical knots can be calculated by counting the subdiagrams of the Gauss diagram with weights, where the weight of a subdiagram is the product of the writhes of all the chords in the subdiagram. For virtual knots, this weight can be strengthened by including the information of chord indices. From this point of view, Proposition 3.6 can be reinterpreted as the signed sum of all subdiagrams which consists of only one chord with a fixed nonzero index is a finite type invariant of degree 1.

It turns out that this idea can be naturally generalized to finite type invariants of higher degrees.

Let $G(K)$ be a Gauss diagram. For any $(x_1, \cdots, x_n)\in(\mathds{Z}-\{0\})^n$ which satisfies $x_1>\cdots>x_n$, we define $C_{(x_1, \cdots, x_n)}$ to be the set of all $n$-chords $\{c_1, \cdots, c_n\}$ in $G(K)$ such that Ind$(c_i)=x_i$.
\begin{theorem}
For each virtual knot $K$, the integer $a_{(x_1, \cdots, x_n)}(K)=\sum\limits_{(c_1, \cdots, c_n)\in C_{(x_1, \cdots, x_n)}}\prod\limits_{i=1}^nw(c_i)$ is a finite type invariant of degree $n$.
\end{theorem}

Note that Proposition 3.6 can be recovered from this theorem by taking $n=1$.
\begin{proof}
First we show that $a_{(x_1, \cdots, x_n)}(K)$ is a virtual knot invariant. For $\Omega_1$, since each chord in $\{c_1, \cdots, c_n\}$ has nonzero index, therefore each $n$-chords in $C_{(x_1, \cdots, x_n)}$ does not contain the chord appears in $\Omega_1$. For $\Omega_2$, let us use $a, b$ to denote the two chords involved in $\Omega_2$. Since $x_i\neq x_j$ if $i\neq j$, hence each $n$-chords in $C_{(x_1, \cdots, x_n)}$ does not contain both $a$ and $b$. If some $n$-chords $\{c_1, \cdots, c_n\}$ contains $a$, for example $c_i=a$. Then the contribution of $\{c_1, \cdots, c_{i-1}, a, c_{i+1}, \cdots, c_n\}$ will be cancelled out by the contribution from $\{c_1, \cdots, c_{i-1}, b, c_{i+1}, \cdots, c_n\}$. For $\Omega_3$, since the index and writhe of each chord are both preserved, it follows directly that $a_{(x_1, \cdots, x_n)}$ is invariant under $\Omega_3$.

Now we show that the degree of $a_{(x_1, \cdots, x_n)}$ is less than $n+1$. It is sufficient to show that for any singular virtual knot diagram $K$ with $n+1$ singular crossings $\{c_1, \cdots, c_{n+1}\}$, we have
\begin{center}
$\sum\limits_{\sigma\in\{0, 1\}^{n+1}}(-1)^{|\sigma|}a_{(x_1, \cdots, x_n)}(K_{\sigma})=0$.
\end{center}
As before here $\sigma$ runs through all $(n+1)$-tuples of zeros and ones, $|\sigma|$ denotes the number of ones in $\sigma$ and $K_{\sigma}$ is obtained from $K$ by replacing $c_i$ with a positive (negative) crossing if the $i$-th position of $\sigma$ is zero (one). Fix a $\sigma\in\{0, 1\}^{n+1}$, choose a $n$-chords $\{c_1', \cdots, c_n'\}$ of $C_{(x_1, \cdots, x_n)}$ in $G(K_{\sigma})$. It is possible that some chords may appear both in $\{c_1, \cdots, c_{n+1}\}$ and $\{c_1', \cdots, c_n'\}$. Without loss of generality we assume that $c_1'=c_1, \cdots, c_i'=c_i$ but $\{c_{i+1}', \cdots, c_n'\}\cap\{c_{i+1}, \cdots, c_{n+1}\}=\emptyset$. Notice that the index of $c_i'$ $(1\leq i\leq n)$ is invariant under switching $c_j$ $(i+1\leq j\leq n+1)$. Now fix the first $i$ positions of $\sigma$ and let the last $n+1-i$ positions of $\sigma$ run over $\{0, 1\}^{n+1-i}$, then the contributions of $\{c_1', \cdots, c_n'\}$ from these $2^{n+1-i}$ virtual knots will cancel out, because half of them have the positive sign and the other half have the negative sign. Repeat this process until the contributions from all $n$-chords have been cancelled out, the desired result follows directly.

To complete the proof we need to show that the degree of $a_{(x_1, \cdots, x_n)}$ is exactly $n$. Let us consider the Gauss diagram in Figure \ref{figure21}, which contains $\sum\limits_{i=1}^n|x_i|+n$ positive chords (Figure \ref{figure21} supposes that $x_1>0$ and $x_n<0$). Notice that Ind$(c_i)=x_i$ and the index of any other chord equals $\pm1$. After replacing the crossings $c_1, \cdots, c_n$ with singular crossings, we obtain a singular virtual knot $K$ which has $n$ singular crossing points. In other words, the virtual knot corresponding to the Gauss diagram described in Figure \ref{figure21} is $K_{(0, \cdots, 0)}$. We claim that for this singular virtual knot $K$ we have $\sum\limits_{\sigma\in\{0, 1\}^{n}}(-1)^{|\sigma|}a_{(x_1, \cdots, x_n)}(K_{\sigma})\neq0$.
\begin{figure}[h]
\centering
\includegraphics{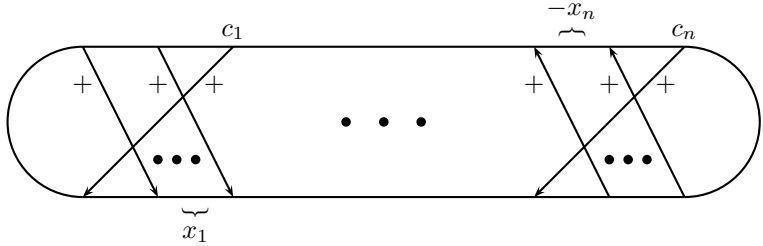}\\
\caption{$K_{(0,\cdots,0)}$}\label{figure21}
\end{figure}
\begin{itemize}
  \item Case 1: for any $1\leq i\leq n$, $x_i\neq\pm1$. If for any $1\leq i<j\leq n$ we have $x_i+x_j\neq0$, then among all $K_{\sigma}$ there exists only one $n$-chords $(c_1, \cdots, c_n)$ which satisfies Ind$(c_i)=x_i$ and it appears in $K_{(0,\cdots,0)}$. In this case we have $\sum\limits_{\sigma\in\{0, 1\}^{n}}(-1)^{|\sigma|}a_{(x_1, \cdots, x_n)}(K_{\sigma})=1$. In general, if there exists a pair of opposite integers $x_i$ and $x_j$. Then another eligible $n$-chords $(c_1, \cdots, c_n)$ also appears in $K_{(0,\cdots, 0, 1, 0, \cdots, 0, 1, 0, \cdots, 0)}$, where the $i$-th and $j$-th positions are 1's. However the contribution from this $n$-chords equals $(-1)^4=1$. We still have $\sum\limits_{\sigma\in\{0, 1\}^{n}}(-1)^{|\sigma|}a_{(x_1, \cdots, x_n)}(K_{\sigma})\neq0$. When there are more pairs of opposite integers in $(x_1, \cdots, x_n)$ one can similarly show that each $n$-chords $(c_1, \cdots, c_n)$ with Ind$(c_i)=x_i$ contributes $+1$ to $\sum\limits_{\sigma\in\{0, 1\}^{n}}(-1)^{|\sigma|}a_{(x_1, \cdots, x_n)}(K_{\sigma})$, the result follows.
  \item Case 2: for some $1\leq i\leq n$, $x_i=1$ but $x_{i+1}\neq-1$. Since we assume that $x_1>\cdots>x_n$, it means that no element of $\{x_1, \cdots, x_n\}$ equals $-1$. In this case, we divide all virtual knots $K_{\sigma}$ into two subsets $K_{(\epsilon_1, \cdots, \epsilon_{i-1}, 0, \epsilon_{i+1}, \cdots, \epsilon_n)}$ and $K_{(\epsilon_1, \cdots, \epsilon_{i-1}, 1, \epsilon_{i+1}, \cdots, \epsilon_n)}$, where $\epsilon_j\in\{0, 1\}$ $(j\neq i)$. A key observation is, when one chooses an eligible $n$-chords from $K_{(\epsilon_1, \cdots, \epsilon_{i-1}, 0, \epsilon_{i+1}, \cdots, \epsilon_n)}$, if the chord with index 1 is not the chord $c_i$ in Figure \ref{figure21}, then one can find a corresponding eligible $n$-chords in $K_{(\epsilon_1, \cdots, \epsilon_{i-1}, 1, \epsilon_{i+1}, \cdots, \epsilon_n)}$. Notice that these two $n$-chords have opposite signs, consequently their contributions kill each other. For this reason we only need to consider the eligible $n$-chords in $K_{(\epsilon_1, \cdots, \epsilon_{i-1}, 0, \epsilon_{i+1}, \cdots, \epsilon_n)}$ where the chord with index 1 is exactly the chord $c_i$ in Figure \ref{figure21}. Similar to Case 1, one can show that each eligible $n$-chords contributes $+1$ to $\sum\limits_{\sigma\in\{0, 1\}^{n}}(-1)^{|\sigma|}a_{(x_1, \cdots, x_n)}(K_{\sigma})$, which implies $\sum\limits_{\sigma\in\{0, 1\}^{n}}(-1)^{|\sigma|}a_{(x_1, \cdots, x_n)}(K_{\sigma})\neq0$.
  \item Case 3: for some $1\leq i\leq n$, $x_{i+1}=-1$ but $x_i\neq1$. The proof is analogous to the proof of Case 2.
  \item Case 4: for some $1\leq i\leq n$, $x_i=1$ and $x_{i+1}=-1$. Now we divide all virtual knots $K_{\sigma}$ into 4 subsets: $K_{(\epsilon_1, \cdots, \epsilon_{i-1}, 0, 0, \epsilon_{i+2}, \cdots, \epsilon_n)}, K_{(\epsilon_1, \cdots, \epsilon_{i-1}, 1, 0, \epsilon_{i+2}, \cdots, \epsilon_n)}, K_{(\epsilon_1, \cdots, \epsilon_{i-1}, 0, 1, \epsilon_{i+2}, \cdots, \epsilon_n)}, K_{(\epsilon_1, \cdots, \epsilon_{i-1}, 1, 1, \epsilon_{i+2}, \cdots, \epsilon_n)}$.

      We write
      \begin{center}
      $\sum\limits_{\sigma\in\{0, 1\}^{n-2}}(-1)^{|\sigma|}a_{(x_1, \cdots, x_n)}(K_{\sigma=(\epsilon_1, \cdots, \epsilon_{i-1}, 0, 0, \epsilon_{i+2}, \cdots, \epsilon_n)})=A_1+A_2+A_3+A_4$,
      \end{center}
      where
      \begin{enumerate}
        \item $A_1$ denotes the contributions from those $n$-chords which contain both $c_i$ and $c_{i+1}$;
        \item $A_2$ denotes the contributions from those $n$-chords which contain $c_i$ but do not contain $c_{i+1}$;
        \item $A_3$ denotes the contributions from those $n$-chords which contain $c_{i+1}$ but do not contain $c_i$;
        \item $A_4$ denotes the contributions from those $n$-chords which contain neither $c_i$ nor $c_{i+1}$.
      \end{enumerate}
      Analogous to Case 1, one can show that $A_1>0$. It is not difficult to observe that
      \begin{center}
      $\sum\limits_{\sigma\in\{0, 1\}^{n-2}}(-1)^{|\sigma|}a_{(x_1, \cdots, x_n)}(K_{\sigma=(\epsilon_1, \cdots, \epsilon_{i-1}, 1, 0, \epsilon_{i+2}, \cdots, \epsilon_n)})=0+A_3+A_3+A_4$,\\
      $\sum\limits_{\sigma\in\{0, 1\}^{n-2}}(-1)^{|\sigma|}a_{(x_1, \cdots, x_n)}(K_{\sigma=(\epsilon_1, \cdots, \epsilon_{i-1}, 0, 1, \epsilon_{i+2}, \cdots, \epsilon_n)})=0+A_2+A_2+A_4$,\\
      $\sum\limits_{\sigma\in\{0, 1\}^{n-2}}(-1)^{|\sigma|}a_{(x_1, \cdots, x_n)}(K_{\sigma=(\epsilon_1, \cdots, \epsilon_{i-1}, 1, 1, \epsilon_{i+2}, \cdots, \epsilon_n)})=A_1+A_3+A_2+A_4$.
      \end{center}
      It follows that
      \begin{center}
      $\sum\limits_{\sigma\in\{0, 1\}^{n}}(-1)^{|\sigma|}a_{(x_1, \cdots, x_n)}(K_{\sigma})=2A_1\neq0$.
      \end{center}
\end{itemize}
\end{proof}

Unfortunately, we would like to remark that this degree $n$ finite type invariant $a_{(x_1, \cdots, x_n)}(K)$ is not a new invariant. It is completely determined by the writhe polynomial. Explicitly, $a_{(x_1, \cdots, x_n)}(K)=\prod\limits_{i=1}^na_{x_i}$ where $a_{x_i}$ is the coefficient of $t^{x_i}$ in the writhe polynomial. In other words, this finite type invariant of degree $n$ is determined by the writhe polynomial, a finite type invariant of degree 1.

In order to obtain some new finite type invariant one needs to ``refine" the chord indices in a subdiagram with more than one chords. For finite type invariants of degree 2, recently the three loop isotopy invariant was introduced by Micah Chrisman and Heather Dye in \cite{Chr2014}. Roughly speaking, the three loop isotopy invariant counts the contributions from the pairs of nonintersecting chords with the same ``triple-index". We end this section with a quick review of this interesting invariant.

Let $G(K)$ be a Gauss diagram and $\mathfrak{c}$ a pair of nonintersecting chords in $G(K)$. We use $T$ to denote the set of all other chords which have nonempty intersection with $\mathfrak{c}$. According to the three possibilities of the orientations of the two chords in $\mathfrak{c}$, the set $T$ can be divided into three subsets $T_1, T_2, T_3$, , see Figure \ref{figure7}. Analogous to the definition of the chord index, for each set $T_i$ we can define an index $t_i$ of $\mathfrak{c}$ which takes values in integers. By taking the absolute value, now $|t_i|$ does not depend on the directions of the chords in $\mathfrak{c}$, hence it is well defined. We call this triple $(|t_1|, |t_2|, |t_3|)$ the \emph{triple-index} of $\mathfrak{c}$.
\begin{figure}[h]
\centering
\includegraphics{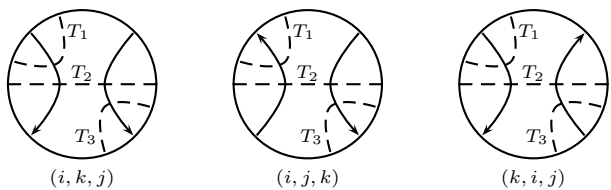}\\
\caption{The set $F_{i, j, k}$}\label{figure7}
\end{figure}

For a fixed triple $(i, j ,k)$ $(i, j, k\geq0$ and $i\neq j\neq k\neq i)$, we define a set $F_{i, j, k}$ which contains all the pairs of nonintersecting chords such that the triple-index of each pair agrees with one of the three cases depicted in Figure \ref{figure7}. Now the \emph{three loop isotopy invariant} of a virtual knot $K$ can be described as
\begin{center}
$\phi(K)=\sum\limits_{\mathfrak{c}\in F_{i, j, k}}w(\mathfrak{c})x^iy^jz^k$,
\end{center}
where $w(\mathfrak{c})$ denotes the product of the writhes of the two chords in $\mathfrak{c}$. We remark that in the original definition of the three loop isotopy invariant, the three loop isotopy invariant $\phi_{i, j, k}(K)$ was defined as the coefficient of $x^iy^jz^k$. Here we rewrite them in the form of a multivariate polynomial.
\begin{theorem}[\cite{Chr2014}]
$\phi(K)$ is a finite type virtual knot invariant of degree 2.
\end{theorem}

It is an interesting question to find some finite type invariants of higher degrees with the help of refined chord indices.

\section{Indexed Jones polynomial}
\subsection{A generalization of the Jones polynomial using chord index}
We know that the Jones polynomial can be naturally extended to virtual knots by taking the approach of  Kauffman bracket. In this section we discuss how to use the chord index to define a sequence of Jones polynomials indexed by nonnegative  integers.

Let $K$ be a virtual knot diagram and $C(K)$ the set of all real crossings of $K$. We define a sequence of subsets of $C(K)$ as below
\begin{center}
$C_n(K)=\{c\in C(K)|\text{Ind}(c)=kn \text{ for some } k\in\mathds{Z}\}$.
\end{center}
Note that every $C_n(K)$ $(n\in\mathds{Z})$ contains all crossing points with index zero. Since $C_n(K)=C_{-n}(K)$, it suffices to consider $n\in\{0, 1, 2, \cdots\}$. For a fixed $n$, we use the standard Kauffman bracket rules to smooth all crossing points in $C_n(K)$, see Figure \ref{figure8}.
\begin{figure}[h]
\centering
\includegraphics{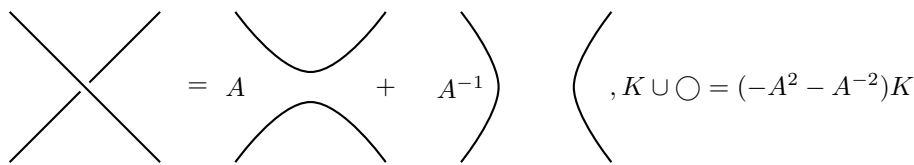}\\
\caption{Smoothing rules}\label{figure8}
\end{figure}

After smoothing all the crossings in $C_n(K)$, we obtain a bracket polynomial $<K>_n$ (here we choose the normalization that the bracket polynomial of the unknot equals one). We define the \emph{Jones polynomial with index $n$} as
\begin{center}
$V_K^n(t)=(-A^3)^{-w(K)}<K>_n|_{A=t^{-\frac{1}{4}}}=(-A^3)^{-w(K)}\sum\limits_SA^{\#_0-\#_1}(-A^2-A^{-2})^{|S|-1}|_{A=t^{-\frac{1}{4}}}$,
\end{center}
where $S$ denotes a state after the smoothing, which is a virtual link diagram, and $|S|$ denotes the number of components of $S$. $\#_0$ and $\#_1$ denote the number of 0-smoothing (the first resolution in Figure \ref{figure8}) and 1-smoothing (the second resolution in Figure \ref{figure8}) taken during the process of smoothing respectively. We remark that for each virtual knot $K$, only finitely many $V_K^n(t)$ are nontrivial. On the other hand, since $C_1(K)=C(K)$, it follows that $V_K^1(t)$ is exactly the classical Jones polynomial of virtual knots.
\begin{theorem}
For each $n\in \{0\}\cup \mathds{N}$, $V_K^n(t)$ is a virtual knot invariant.
\end{theorem}
\begin{proof}
It is sufficient to show that $V_K^n(t)$ is invariant under each generalized Reidemeister move in Figure \ref{figure1}. Since no real crossing point is involved in $\Omega_1', \Omega_2', \Omega_3'$, we only need to consider the rest four moves.
\begin{itemize}
  \item $\Omega_1:$ the only crossing point involved in $\Omega_1$ has index zero, hence for each $n$ this crossing point must be smoothed. Similar to the classical case, $(-A^3)^{-w(K)}$ is designed to nullify this change.
  \item $\Omega_2:$ the two crossings involved in $\Omega_2$ have the same index. If this index is not a multiple of $n$, then of course $V_K^n(t)$ is preserved. If the index of these two crossing points is a multiple of $n$, then both of them will be smoothed. The rule that adding a circle disjoint from the rest diagram multiplies the bracket by $-A^2-A^{-2}$ will nullify this change.
  \item $\Omega_3:$ we know that the indices of the three crossing points are preserved respectively under $\Omega_3$. A key fact is, on both sides of $\Omega_3$ the index of one crossing equals the sum of the indices of other two crossings \cite{Che2016}. If none index of these three crossing points is a multiple of $n$, then the result follows obviously. Otherwise either only one of them will be smoothed, or all of them will be smoothed. It is easy to observe that in both cases the number of components of the state is invariant.
  \item $\Omega_3^v:$ there is nothing need to prove in this case.
\end{itemize}
\end{proof}

Recall that the \emph{span} of a polynomial is the difference between the highest degree and the lowest degree. It is well known that the span of the Jones polynomial provides a lower bound for the crossing number of classical knots. The following result is a generalization of this fact.
\begin{proposition}
Let $K$ be a virtual knot diagram, then $|C_n(K)|\geq\text{span}V_K^n(t)$, here $|C_n(K)|$ denotes the cardinality of $C_n(K)$.
\end{proposition}
\begin{proof}
Assume $|C_n(K)|=m$. Let $S_0(K) (S_1(K))$ be the state obtained from $K$ by taking $0 (1)$-smoothing on all the crossing points in $C_n(K)$. According to the definition of $V_K^n(t)$, the potential highest degree and lowest degree of $V_K^n(t)$ are equal to the highest degree and the lowest degree provided by $S_0(K)$ and $S_1(K)$ respectively. It is easy to conclude that span$V_K^n(t)\leq \frac{1}{4}(2m+2|S_0(K)|+2|S_1(K)|-4)$. We claim that $|S_0(K)|+|S_1(K)|\leq m+2$, then result follows directly.

First notice that we can assume $|C(K)|=|C_n(K)|$, i.e. all real crossing points of $K$ belong to the set $C_n(K)$. If not, we can replace all other real crossing points with virtual crossing points. Although $V_K^n(t)$ is not preserved (since the writhe of the diagram may be changed), the span of $V_K^n(t)$ is invariant. Therefore from now on we assume that $|C(K)|=|C_n(K)|=m$, i.e. every real crossing point will be smoothed. The following proof is similar to the proof of the Dual State Lemma in \cite{Kau1987}. For the completeness, we still sketch it here.

If the number of the real crossing points in $K$ equals 0 or 1, it is easy to get the desired result. Suppose that for all virtual knot diagrams with $m-1$ real crossing points the result is true, it suffices to prove that the result is also true for any virtual knot diagram $K$ which contains $m$ real crossing points. Choose a real crossing point of $K$, say $c$, without loss of generality we assume that taking 0-smoothing at $c$ yields a new virtual knot diagram $K'$. Due to the hypothesis, now we have $|S_0(K')|+|S_1(K')|\leq m+1$. Note that $|S_0(K)|=|S_0(K')|$ and $||S_1(K)|-|S_1(K')||\leq1$, it follows that $|S_0(K)|+|S_1(K)|\leq m+2$.
\end{proof}

In virtual knot theory, the \emph{real crossing number} $c_r(K)$ is the minimal number of the real crossing points among all virtual knot diagrams of $K$. When $K$ is a classical knot, this coincides with the crossing number of $K$ \cite{Man2013}. A natural refinement of $c_r(K)$ is the minimal number of crossing points with index $n$ among all knot diagrams. We use $c_r^n(K)$ to denote it. As we mentioned in Section 3, if $n\neq 0$ the absolute value of the coefficient of $t^n$ in the writhe polynomial provides a lower bound for $c_r^n(K)$. However the writhe polynomial can not tell us anything about $c_r^0(K)$. Now Proposition 4.2 tells us it is possible to use $V_K^0(t)$ to give a lower bound for $c_r^0(K)$. Here is an example.
\begin{example}
Let us consider the virtual knot $K$ described in Figure \ref{figure10}. The four real crossing points $a, b, c, d$ have indices $0, 1, 0, -1$ respectively. The writhe polynomial of it equals $t^1+t^{-1}$, which implies $c_r^1(K)=c_r^{-1}(K)=1$. On the other hand, one computes $V_K^0(t)=-t^{4}+t^{3}+t^{\frac{5}{2}}$, hence we have span$V_K^0(t)=\frac{3}{2}$. Therefore we conclude that $c_r^0(K)=2$.
\begin{figure}[h]
\centering
\includegraphics{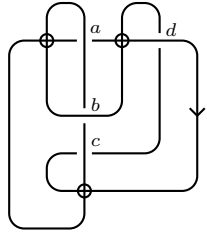}\\
\caption{A virtual knot with $c_r^0=2, c_r^1=c_r^{-1}=1$}\label{figure10}
\end{figure}
\end{example}

\subsection{Some generalizations of the indexed Jones polynomial}
Recall our definition of the indexed Jones polynomial, for each state we only extract the information of the number of components. In classical case, each state is a disjoint union of circles. The only useful information is the number of these circles. But in our case, each state is a virtual link diagram. A natural idea is to extract more information from this virtual link diagram rather than just counting the number of components.

Recall that given a virtual knot diagram $K$, the corresponding flat virtual knot $F(K)$ is obtained from $K$ by replacing all real crossing points with flat crossing points. Now we define a graphical modification of the indexed Jones polynomial as follows
\begin{center}
$\mathfrak{V}_K^n(t)=(-A^3)^{-w(K)}\sum\limits_SA^{\#_0-\#_1}F(S)|_{A=t^{-\frac{1}{4}}}$.
\end{center}
Let $FVL$ be the free abelian group generated by all flat virtual links. Then $\mathfrak{V}_K^n(t)$ takes values in $FVL[t^{\frac{1}{2}}, t^{-\frac{1}{2}}]$, i.e. each $\mathfrak{V}_K^n(t)$ is a Laurent polynomial with coefficients in $FVL$. Note that $F(K\cup O)=(-A^2-A^{-2})F(K)$, where $O$ denotes a circle which is disjoint from $K$.
\begin{theorem}
$\mathfrak{V}_K^n(t)$ is a virtual knot invariant.
\end{theorem}
\begin{proof}
The proof is almost the same with that of Theorem 4.1. We only mention that the map $F$ will be used in the case when only one crossing point in $\Omega_3$ is smoothed.
\end{proof}
\begin{remark}
This graphical modification of the indexed Jones polynomial is motivated by the parity bracket polynomial proposed by Manturov in \cite{Man2010}. Actually, $\mathfrak{V}_K^2(t)$ is exactly Manturov's parity bracket polynomial. The parity arrow polynomial was defined by Kaestner and Kauffman in \cite{Kae2012}, which combines the idea of parity and arrow polynomial. We remark that with the help of the chord index, one can similarly define the indexed Miyazawa/arrow polynomial, which also provides a generalization of the indexed Jones polynomial.
\end{remark}

\section{Indexed quandle and its applications}
\subsection{A quick review of the quandle structure}
To set the stage, we recall some basic notions of quandle and its generalizations.
\begin{definition}
A \emph{quandle} $Q$ is a set with a binary operation $\ast: Q\times Q\rightarrow Q$, which satisfy
\begin{enumerate}
  \item $\forall a\in Q$, $a\ast a=a$;
  \item $\forall b, c\in Q$, there exists a unique $a\in Q$ such that $a\ast b=c$;
  \item $\forall a, b, c\in Q$, $(a\ast b)\ast c=(a\ast c)\ast(b\ast c)$.
\end{enumerate}
\end{definition}

Here we list some examples of quandles.
\begin{itemize}
  \item For any nonempty set $Q$ one can define the \emph{trivial quandle} by setting $a\ast a=a$ for any $a\in Q$;
  \item A conjugacy class of a group with quandle operation $a\ast b=b^{-1}ab$ is called a \emph{conjugation quandle};
  \item Consider the unit sphere $S^n$ in $\mathds{R}^{n+1}$, the operation $a\ast b=2(a\cdot b)b-a$ provides a quandle structure on $S^n$. Here $\cdot$ denotes the inner product of $\mathds{R}^{n+1}$.
\end{itemize}

The notion of quandle was defined by Joyce in \cite{Joy1982} (and independently by Matveev in \cite{Mat1984} with the name \emph{distributive groupoid}). Given a classical knot diagram $K$, the \emph{knot quandle} (also called \emph{fundamental quandle}) $Q(K)$ is generated by the arcs of $K$ and subject to the relations indicated in Figure \ref{figure11}. It is well known that the knot quandle distinguishes knots up to mirrors and reverses. However similar to the knot group, in general the knot quandle is difficult to deal with. In practise, a more easily computable invariant is the number of quandle homomorphisms from $Q(K)$ to a fixed finite quandle $Q$. We call it the \emph{coloring invariant} and use $Col_Q(K)$ to denote it. From the viewpoint of knot diagram, a coloring assigns to each arc of the knot diagram an element of $Q$ such that at each crossing the coloring rule indicated in Figure \ref{figure11} is satisfied.
\begin{figure}[h]
\centering
\includegraphics{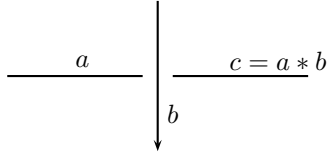}\\
\caption{The coloring rule at each crossing}\label{figure11}
\end{figure}

The knot quandle can be similarly defined for virtual knots, which is also an invariant. However the knot quandle is not good at distinguishing virtual knots. For example, the knot quandle of the virtual trefoil described in Figure \ref{figure2} is trivial. For this reason, R. Fenn, M. Jordan-Santana and Louis H. Kauffman \cite{Fen2004} introduced a more general algebraic structure, say the biquandle. Note that the definition used here was suggested in \cite{Kae2015}, which is a bit different from the original definition introduced in \cite{Fen2004}. But it is easy to see that they are essentially the same.
\begin{definition}
A \emph{biquandle} $BQ$ is a set with two binary operations $\ast, \circ: BQ\times BQ\rightarrow BQ$ such that the following axioms are satisfied
\begin{enumerate}
  \item $\forall x\in BQ, x\ast x=x\circ x$;
  \item $\forall x, y\in BQ$, there are unique $z, w\in BQ$ such that $z\ast x=y$ and $w\circ x=y$, and the map $S: (x, y)\rightarrow(y\circ x, x\ast y)$ is invertible;
  \item $\forall x, y, z\in BQ$, we have
  \begin{center}
  $(z\circ y)\circ(x\ast y)=(z\circ x)\circ(y\circ x)$,\\$(y\circ x)\ast(z\circ x)=(y\ast z)\circ(x\ast z)$,\\$(x\ast y)\ast(z\circ y)=(x\ast z)\ast(y\ast z)$.
  \end{center}
\end{enumerate}
\end{definition}

The following is a simple observation.
\begin{lemma}
Let $(BQ, \ast, \circ)$ be a biquandle and $x, y$ are two elements of $BQ$, if $x\ast y=y\circ x$ then $x=y$.
\end{lemma}
\begin{proof}
Taking $z=y$, now the first condition in the third axiom of biquandle becomes
\begin{center}
$(y\circ y)\circ(x\ast y)=(y\circ x)\circ(y\circ x)$.
\end{center}
Since $x\ast y=y\circ x$, it follows that
\begin{center}
$(y\circ y)\circ(y\circ x)=(y\circ x)\circ(y\circ x)$.
\end{center}
Recall that $-\circ(y\circ x)$ is invertible, hence we obtain $y\circ y=y\circ x=x\ast y$. Together with the first axiom $y\circ y=y\ast y$, we have $y\ast y=x\ast y$. The desired result follows directly because $-\ast y$ is invertible.
\end{proof}

Later the biquandle structure was extended by Kauffman and Manturov to the virtual biquandle \cite{Kau2005} by adding a relation at each virtual crossing point. The readers are referred to \cite{Car2012} for more details of quandle ideas. We will come back to the biquandle structure in Section 6.

\subsection{Indexed quandle}
In this subsection we plan to give a generalized quandle structure by using the chord index. The main idea is similar to the parity biquandle introduced in \cite{Kae2014} and studied in \cite{Kae2015}, where the parity was used to generalize the quandle structure. What we want to do is replacing the parity with chord index. We note that although we only discuss the indexed quandle in this paper, the technique used here can be similarly used to define an indexed biquandle.
\begin{definition}
An \emph{indexed quandle} is a set $IndQ$ with a sequence of binary operations $\ast_i: IndQ\times IndQ\rightarrow IndQ$ $(i\in \mathds{Z})$ which satisfy the following axioms:
\begin{enumerate}
  \item $\forall a\in IndQ, a\ast_0 a=a$;
  \item $\forall b, c\in IndQ$ and $i\in \mathds{Z}$, there is a unique $a\in IndQ$ such that $a\ast_i b=c$;
  \item $\forall a, b, c\in IndQ$ and $i, j\in\mathds{Z}, (a\ast_i b)\ast_j c=(a\ast_j c)\ast_i(b\ast_{j-i}c)$.
\end{enumerate}
\end{definition}

We remark that each indexed quandle $(IndQ, \ast_i)$ includes a quandle $(IndQ, \ast_0)$. The followings are some examples of the indexed quandles.
\begin{itemize}
  \item Let $(Q, \ast)$ be a quandle, then $Q$ can be thought of as an indexed quandle if we define $\ast_i=\ast$ for any $i\in \mathds{Z}$.
  \item Let $(Q, \ast)$ be a quandle, another way to regard $Q$ as an indexed quandle is introducing $\ast_0=\ast$ and $a\ast_{i(\neq0)} b=a$ for any $a, b\in Q$.
  \item Let $X=\mathds{Z}[t, t^{-1}]$, one defines $a\ast_i b=ta+(1-t)b+i$ for any $a, b\in\mathds{Z}[t, t^{-1}]$.
  \item Let $G$ be a group, for any $\phi\in Aut(G)$ and $z\in Z(G)$ (the center of $G$), $G$ can be regarded as an indexed quandle with operations $a\ast_i b=\phi(ab^{-1})bz^i$.
\end{itemize}

For a given virtual knot diagram $K$, we define the \emph{indexed knot quandle} $IndQ_K$ to be the indexed quandle generated by the arcs of $K$, and each real crossing point gives a relation which depends on the index, see Figure \ref{figure12}.
\begin{figure}[h]
\centering
\includegraphics{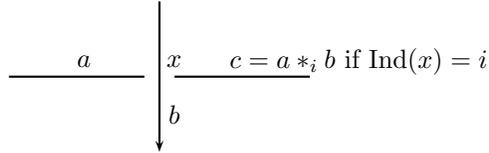}\\
\caption{The indexed coloring rule at each crossing}\label{figure12}
\end{figure}

Notice that when $K$ is a classical knot diagram, then the index of each crossing is zero, therefore in this case the indexed knot quandle reduces to the classical knot quandle.

\begin{theorem}
The indexed knot quandle is preserved under the generalized Reidemeister moves, hence it is a virtual knot invariant.
\end{theorem}
\begin{proof}
According to the definition of the indexed knot quandle, it is sufficient to verify its invariance under $\Omega_1, \Omega_2, \Omega_3$. Recall that the crossing involved in $\Omega_1$ has index zero and the two crossing points involved in $\Omega_2$ have the same index, it is easy to conclude the invariance of $IndQ_K$ under $\Omega_1, \Omega_2$ from the first and second axiom respectively.

For the third Reidemeister move $\Omega_3$, note that the corresponding crossing points on the two sides of $\Omega_3$ have the same index. We assume Ind$(x)$=Ind$(x')=i$, Ind$(y)$=Ind$(y')=j$ and Ind$(z)$=Ind$(z')=k$. Similar to the proof of Lemma 4.1 in \cite{Che2016}, it is not difficult to show that $i=j+k$. Together with the third axiom, the invariance of $IndQ_K$ can be obtained from Figure \ref{figure13} directly.
\end{proof}
\begin{figure}[h]
\centering
\includegraphics{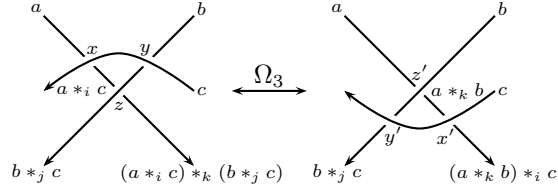}\\
\caption{The invariance of $IndQ_K$ under $\Omega_3$}\label{figure13}
\end{figure}

\begin{corollary}
Given a finite indexed quandle $(Q, \ast_i)$, $Col_{(Q, \ast_i)}(K)=|Hom(IndQ_K, (Q, \ast_i))|$ is a virtual knot invariant.
\end{corollary}

Notice that if $K$ is a classical knot, then for any indexed quandle $(Q, \ast_i)$ we have
\begin{center}
$Col_{(Q, \ast_i)}(K)=|Hom(IndQ_K, (Q, \ast_i))|=|Hom(Q_K, (Q, \ast_0))|=Col_{(Q, \ast_0)}(K)$.
\end{center}
Therefore for classical knots, there always exist some trivial colorings. However this is not the case for virtual knots.

\begin{example}
Consider the virtualization of the classical trefoil knot in Figure \ref{figure14}. There are three real crossing points $\{a, b, c\}$ with Ind$(a)=-2$, Ind$(b)=2$ and Ind$(c)=0$. If we use the dihedral quandle $D_3=\{0, 1, 2\}$ with operation $i\ast j=2j-i$ (mod 3), then this virtualized trefoil knot has the same coloring invariant with the classical trefoil knot. If we use the indexed quandle $IndD_3=\{0, 1, 2\}$ with operations $i\ast_k j=2j-i+k$ (mod 3), as we discussed above the classical trefoil knot still has 9 different colorings. However by solving the following equations in $\mathds{Z}_3$
\begin{equation*}
\left\{
\begin{array}{c}
2y-x+2=z\\
2z-x+0=y\\
2x-z-2=y\\
\end{array}
\right.
\end{equation*}
one finds that there is no solution, it means that the coloring invariant of the virtualized trefoil knot is 0.
\end{example}
\begin{figure}[h]
\centering
\includegraphics{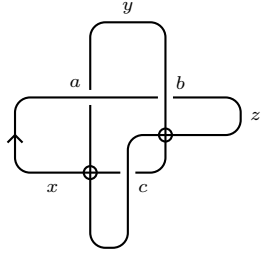}\\
\caption{The virtualization of the trefoil knot}\label{figure14}
\end{figure}

\subsection{An enhancement of the coloring invariant via indexed quandle 2-cocycles}
The coloring invariant $Col_Q(K)$ has many applications in knot theory. For example, $Col_Q(K)$ can be used to provide lower bounds for braid index, tunnel number and unknotting number \cite{Cla2014}. However, there also exists some disadvantages to $Col_Q(K)$. For instance, for any finite quandle $Q$, $Col_Q(K)$ can not distinguish the left-hand trefoil from the right-hand trefoil, since the trefoil knot is invertible. In \cite{Car2003}, J. S. Carter, D. Jelsovsky, S. Kamada, L. Langford and M. Saito introduced the quandle (co)homology theory. In particular, with a given quandle 2(3)-cocycle one can define an enhancement of the (shadow) coloring invariant, say the quandle cocycle invariant. There are many examples which show that these cocycle invariants are more powerful than the original coloring invariants. As an example, for some suitably chosen quandle and quandle 3-cocycle the cocycle invariant takes different values on the left-hand trefoil and the right-hand trefoil \cite{Rou2000}. Following the main idea of the quandle cocycle invariant we want to define the notion of indexed quandle 2-cocycle in this subsection. Analogous to the quandle 2-cocycle, each indexed quandle 2-cocycle also can be used to define a generalized invariant of $Col_{(Q, \ast_i)}(K)$.
\begin{definition}
Let $(Q, \ast_i)$ be an indexed quandle and $A$ an abelian group, we call a map $\phi: Q\times Q\rightarrow A$ an \emph{indexed quandle 2-cocycle} if for any $a, b, c\in Q$ and $i, j\in \mathds{Z}$ it satisfies
\begin{center}
$\phi(a, c)^{-1}\phi(a\ast_{i-j}b, c)\phi(a, b)\phi(a\ast_i c, b\ast_j c)^{-1}=1$ and $\phi(a, a)=1$.
\end{center}
\end{definition}

For a given virtual knot diagram $K$ and a finite indexed quandle $(Q, \ast_i)$, we choose a coloring $f: IndQ_K\rightarrow (Q, \ast_i)$. For a crossing $x$ with index $i$, if the three arcs around $x$ have colors $a, b$ and $a\ast_i b$ (see Figure \ref{figure12}) then we associate a weight $B(x, f)=\phi(a, b)^{w(x)}$ to the crossing point $x$. As before here $w(x)$ denotes the writhe of $x$. Now we define the \emph{indexed quandle 2-cocycle invariant} (associated with $\phi$) to be
\begin{center}
$\Phi_{\phi}(K)=\sum\limits_f\prod\limits_xB(x, f)$,
\end{center}
where $f$ runs over all homomorphisms from $IndQ_K$ to $(Q, \ast_i)$, and $x$ runs over all real crossing points of $K$. Obviously if $\phi$ sends each element of $Q\times Q$ to the identity element, then $\Phi_{\phi}(K)$ reduces to the coloring invariant $Col_{(Q, \ast_i)}(K)$.

\begin{theorem}
$\Phi_{\phi}(K)$ is a virtual knot invariant.
\end{theorem}
\begin{proof}
Since the crossing point involved in $\Omega_1$ has index zero, together with $\phi(a, a)=1$ for any $a\in Q$, it follows that $\Phi_{\phi}(K)$ is invariant under $\Omega_1$. For $\Omega_2$, notice that the two crossing points in $\Omega_2$ have opposite signs, hence the contributions from these two crossing points cancel out. The invariance of $\Phi_{\phi}(K)$ under $\Omega_3$ can be read directly from Figure \ref{figure13}.
\end{proof}

We give an example to illustrate that $\Phi_{\phi}(K)$ is strictly stronger that the coloring invariant $Col_{(Q, \ast_i)}(K)$.

\begin{example}
In this example we use $K$ and $K'$ to denote the virtual trefoil knot depicted in Figure \ref{figure2} and the virtualization of the classical trefoil knot depicted in Figure \ref{figure14}. Consider a set $Q=\{0, 1\}$ with operations $a\ast_i b=a+i$ (mod 2). One easily finds that $(Q, \ast_i)$ is an indexed quandle. For the coloring invariant, we have
\begin{center}
$|Col_{(Q, \ast_i)}(K)|=|Col_{(Q, \ast_i)}(K')|=2$.
\end{center}
In order to define an indexed quandle 2-cocycle invariant we choose $A=\mathds{Z}_2$ and introduce an indexed quandle 2-cocycle $\phi$, which is defined by $\phi(0, 0)=\phi(1, 1)=0$ and $\phi(0, 1)=\phi(1, 0)=1$. Now we have
\begin{center}
$\Phi_{\phi}(K)=1+1$ but $\Phi_{\phi}(K')=0+0$.,
\end{center}
which means that $K$ and $K'$ have the same coloring invariants but different cocycle invariants.
\end{example}

We remark that the indexed quandle 2-cocycle invariant can be extended by replacing $\phi$ with a sequence of homomorphisms $\phi_i$ $(i\in \mathds{Z})$. More precisely, we use $\psi$ to denote a sequence of $\{\phi_i\}_{i\in\mathds{Z}}$ where each $\phi_i$ represents a map from $Q\times Q$ to $A$. We say $\psi$ is a \emph{generalized indexed quandle 2-cocycle} if for any $a, b, c\in Q$ and $i, j\in \mathds{Z}$ we have
\begin{center}
$\phi_i(a, c)^{-1}\phi_i(a\ast_{i-j}b, c)\phi_{i-j}(a, b)\phi_{i-j}(a\ast_i c, b\ast_j c)^{-1}=1$ and $\phi_0(a, a)=1$.
\end{center}
With a fixed generalized indexed quandle 2-cocycle $\psi=\{\phi_i\}_{i\in\mathds{Z}}$ and a coloring $f$, we associate a weight $\mathfrak{B}(x, f)=\phi_i(a, b)^{w(x)}$ to the crossing point $x$ in Figure \ref{figure12}. Now we define the \emph{generalized indexed quandle 2-cocycle invariant} as follows
\begin{center}
$\Psi_{\psi}(K)=\sum\limits_f\prod\limits_{x}\mathfrak{B}(x, f)$,
\end{center}
where the product takes over all crossing points and the sum takes over all colorings. In the same way, one can prove that $\Psi_{\psi}(K)$ is a virtual knot invariant.

We end this section with a simple example of the generalized indexed quandle 2-cocycle invariant. Consider the indexed quandle which consists of one element $a$ and choose the abelian group $A=\mathds{Z}[t, t^{-1}]$. We define a generalized indexed quandle 2-cocycle $\psi=\{\phi_i\}$ by letting $\phi_i(a, a)=t^i$ $(i\neq 0)$ and $\phi_0(a, a)=0$. Note that these is only one coloring. Now the generalized indexed quandle 2-cocycle invariant can be read as
\begin{center}
$\Psi_{\psi}(K)=\sum\limits_{\text{Ind}(x)\neq0}w(x)t^{\text{Ind}(x)}=W_K(t)$,
\end{center}
which is exactly the writhe polynomial we discussed in Section 3. It means that the writhe polynomial can be understood as a special case of the generalized indexed quandle cocycle invariant.

\subsection{Abelian extensions of indexed quandle by 2-cocycles}
In group theory, it is well known that there is a one to one correspondence between the set of isomorphism classes of central extensions of $G$ by $A$ and the cohomology group $H^2(G, A)$. The analogous relation between quandle extensions and quandle 2-cocycles was given by J. S. Carter et al in \cite{Car2003E}. In this subsection we want to explore the relation between the extensions of indexed quandles and the generalized indexed quandle 2-cocycles.

For an indexed quandle $(Q, \ast_i)$, an abelian group $A$ and a sequence of maps $\psi=\{\phi_i:Q\times Q\rightarrow A, i\in \mathds{Z}\}$, we define a set $E(Q, A, \psi)=A\times Q$ equipped with a sequence of binary operations
\begin{center}
$(a_1, x_1)\ast_i(a_2, x_2)=(a_1\phi_i(x_1, x_2), x_1\ast_ix_2)$.
\end{center}
The following proposition says $E(Q, A, \psi)$ is an indexed quandle if and only if $\psi$ is a generalized indexed quandle 2-cocycle. In this case, we say the set $E(Q, A, \psi)$ is an \emph{abelian extension} of $(Q, \ast_i)$.
\begin{proposition}
$E(Q, A, \psi)$ is an indexed quandle if and only if $\psi$ is a generalized indexed quandle 2-cocycle.
\end{proposition}
\begin{proof}
We assume $\psi$ is a generalized indexed quandle 2-cocycle. Recall that this means that for any $x_1, x_2, x_3$ of $Q$ we have
\begin{center}
$\phi_i(x_1, x_3)^{-1}\phi_i(x_1\ast_{i-j}x_2, x_3)\phi_{i-j}(x_1, x_2)\phi_{i-j}(x_1\ast_i x_3, x_2\ast_j x_3)^{-1}=1$ and $\phi_0(x_1, x_1)=1$.
\end{center}
First note that
\begin{center}
$(a_1, x_1)\ast_0(a_1, x_1)=(a_1\phi_0(x_1, x_1), x_1\ast_0x_1)=(a_1, x_1)$,
\end{center}
and
\begin{center}
$(a_3, x_3)\ast_i^{-1}(a_2, x_2)=(a_3(\phi_i(x_3\ast_i^{-1}x_2, x_2))^{-1}, x_3\ast_i^{-1}x_2)$.
\end{center}
Next it suffices to prove
\begin{center}
$((a_1, x_1)\ast_i(a_2, x_2))\ast_j(a_3, x_3)=((a_1, x_1)\ast_j(a_3, x_3))\ast_i((a_2, x_2)\ast_{j-i}(a_3, x_3))$.
\end{center}
One computes
\begin{flalign*}
&((a_1, x_1)\ast_i(a_2, x_2))\ast_j(a_3, x_3)\\
=&(a_1\phi_i(x_1, x_2), x_1\ast_ix_2))\ast_j(a_3, x_3)\\
=&(a_1\phi_i(x_1, x_2)\phi_j(x_1\ast_ix_2, x_3), (x_1\ast_ix_2)\ast_jx_3)\\
=&(a_1\phi_j(x_1, x_3)\phi_i(x_1\ast_jx_3, x_2\ast_{j-i}x_3), (x_1\ast_jx_3)\ast_i(x_2\ast_{j-i}x_3))\\
=&(a_1\phi_j(x_1, x_3), x_1\ast_jx_3)\ast_i(a_2\phi_{j-i}(x_2, x_3), x_2\ast_{j-i}x_3)\\
=&((a_1, x_1)\ast_j(a_3, x_3))\ast_i((a_2, x_2)\ast_{j-i}(a_3, x_3)).
\end{flalign*}

Conversely, if $E(Q, A, \psi)$ is an indexed quandle one can similarly prove that $\psi$ satisfies the 2-cocycle conditions.
\end{proof}

\section{What is a chord index?}
In previous sections we have listed several applications of the chord index in virtual knot theory. A natural question is, is it possible to define the chord index in a more general manner? Or more generally, what is a chord index indeed? Motivated by the parity axioms proposed by Manturov in \cite{Man2010}, here we introduce the chord index axioms in virtual knot theory.
\begin{definition}
Assume we have a rule which assigns an index (e.g. an integer, a polynomial, a group, etc.) to each real crossing point in a virtual link diagram. We say this rule satisfies the \emph{chord index axioms} if it satisfies the following conditions:
\begin{enumerate}
  \item The real crossing point involved in $\Omega_1$ has a fixed index.
  \item The two real crossing points involved in $\Omega_2$ have the same index.
  \item There is a natural 1-1 correspondence between the real crossing points involved in $\Omega_3$. The corresponding real crossing points have the same index.
  \item The index of the real crossing point involved in $\Omega_3^v$ is preserved under $\Omega_3^v$.
  \item The indices of all real crossing points which are not involved in a generalized Reidemeister move are preserved under this generalized Reidemeister move.
\end{enumerate}
\end{definition}

It is easy to observe that our chord index defined in Section 2 satisfies all chord index axioms for virtual knot diagrams. As a generalization of the chord index, in \cite{Che2016} we introduced the index function, which also satisfies all chord index axioms above. In this section we would like to provide a general construction of chord index which satisfies all chord index axioms above. Note that the chord index and the index function only can be defined for the real crossing points in a virtual knot diagram. However the following manner also can be used to define the chord index for real crossing points in a virtual link diagram.

In our original idea of the chord index \cite{Che2014}, the chord index is deduced from a $\mathds{Z}$-coloring on the semiarcs of a knot diagram. Later in \cite{Kau2013}, Kauffman introduced the notion of flat biquandle, which provides a more general algorithm for the colorings. In particular, Kauffman proved that essentially the coloring used in \cite{Che2014} is the unique affine linear flat biquandle. The main result of this section is to define a general chord index via biquandles.

Recall that a biquandle is a set equipped with two binary operations $\ast$ and $\circ$ (see Definition 5.2). Similar to the knot quandle, one can define the \emph{knot biquandle} $BQ_K$, which is generated by all the semiarcs (a segment of the diagram from a real crossing to the next real crossing) of a virtual knot diagram but now each real crossing point offers two relations, see Figure \ref{figure15}. The axioms of the biquandle guarantees the invariance of $BQ_K$ under the generalized Reidemeister moves. More precisely, the first axiom and Lemma 5.3 can be used to prove the invariance under $\Omega_1$. The seconde Reidemeister move $\Omega_2$ follows from the second axiom. See Figure \ref{figure19} for the invariance of $BQ_K$ under $\Omega_3$. Similar to the knot quandle, for any finite biquandle $BQ$, one can define the coloring invariant $Col_{BQ}(K)$ to be $|Hom(BQ_K, BQ)|$.
\begin{figure}[h]
\centering
\includegraphics{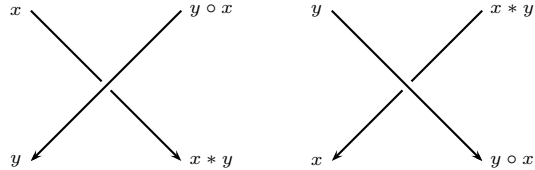}\\
\caption{The coloring rule of biquandle}\label{figure15}
\end{figure}
\begin{figure}[h]
\centering
\includegraphics{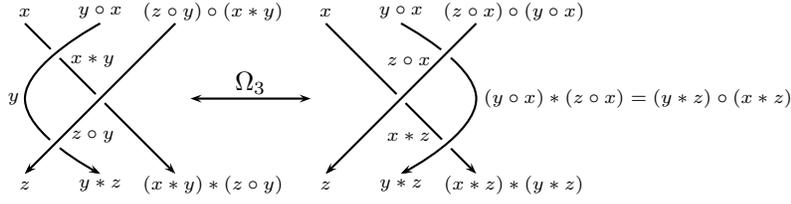}\\
\caption{Biquandle coloring under $\Omega_3$}\label{figure19}
\end{figure}

Fix a finite biquandle $BQ$ and an abelian group $A$, a map $\phi: BQ\times BQ\rightarrow A$ is called a \emph{(reduced) biquandle 2-cocycle} if for any $x, y, z\in BQ$ we have
\begin{center}
$\phi(x, x)=1$ and $\phi(x, y)\phi(y, z)\phi(x\ast y, z\circ y)=\phi(x\ast z, y\ast z)\phi(y\circ x, z\circ x)\phi(x, z)$.
\end{center}
For the sake of convenience, we introduce the \emph{universal 2-cocycle group} of $BQ$, which can be defined as the abelianization of
\begin{center}
$G_{BQ}=\langle(x, y)\in BQ\times BQ|(x, x)=1, (x, y)(y, z)(x\ast y, z\circ y)=(x\ast z, y\ast z)(y\circ x, z\circ x)(x, z)\rangle$.
\end{center}
We name it in this way since each homomorphism $\rho: G_{BQ}/[G_{BQ}, G_{BQ}]\rightarrow A$ provides a biquandle 2-cocycle. More precisely, for any biquandle 2-cocycle $\phi: BQ\times BQ\rightarrow A$ there is a map $\rho: G_{BQ}/[G_{BQ}, G_{BQ}]\rightarrow A$ such that the following diagram commutes

\hspace{5cm} \xymatrix{
  BQ\times BQ \ar[dr]_{\phi} \ar[r]^{i}
                & G_{BQ}/[G_{BQ}, G_{BQ}] \ar[d]^{\rho}  \\
                & A            }\\
here $i$ denotes the quotient map from $BQ\times BQ$ to $G_{BQ}/[G_{BQ}, G_{BQ}]$.

Let us consider another group
\begin{center}
$\mathfrak{G}_{BQ}=\langle(x, y)\in BQ\times BQ|(x, x)=1, (x, y)=(x\ast z, y\ast z), (y, z)=(y\circ x, z\circ x), (x, z)=(x\ast y, z\circ y)\rangle$.
\end{center}
In general $\mathfrak{G}_{BQ}$ is not an abelian group. Let $K$ be a virtual knot diagram and $f$ a coloring, we associate a weight $\mathfrak{W}_f=(x, y)\in \mathfrak{G}_{BQ}$ to the two crossing points in Figure \ref{figure15}. The main difference between $G_{BQ}$ and $\mathfrak{G}_{BQ}$ is, $G_{BQ}$ requires that the sum of the contributions coming from the three crossing points on the left side of $\Omega_3$ equals the sum of the contributions provided by the three crossing points on the right side. However for $\mathfrak{G}_{BQ}$, we require the contribution from each crossing point involved in $\Omega_3$ is preserved under $\Omega_3$, see Figure \ref{figure19}. Now we define the \emph{index} (associated to $BQ$) of a crossing point to be $\sum\limits_f\mathfrak{W}_f\in\mathds{Z}\mathfrak{G}_{BQ}$. Notice that the index does not depend on the choice of $f$. If there exists a homomorphism $\rho$ from $\mathfrak{G}_{BQ}$ to some other group $A$, then we obtain an induced chord index $\rho(\sum\limits_f\mathfrak{W}_f)\in\mathds{Z}A$.

\begin{example}[\cite{Kau2013}]
Let $X=(\mathds{Z}, \ast, \circ)$ be a biquandle, where $x\ast y=x\circ y=x+1$. For each virtual knot diagram there exist infinitely many colorings. In particular, if one chooses a coloring $f$ , then any other coloring can be obtained from $f$ by adding an integer to the assigned number on each semiarc of $K$. Consider a map $\rho: \mathds{Z}\times \mathds{Z}\rightarrow \mathds{Z}$ defined by $\rho(x, y)=y-x$. One can naturally extend this map to a homomorphism from $\mathfrak{G}_X$ to $\mathds{Z}[\mathds{Z}]$. We still use $\rho$ to denote it. Now the induced chord indices of the two crossing points depicted in Figure \ref{figure15} are both equal to $\sum\limits_{\mathds{Z}}\rho(x, y)=\sum\limits_{\mathds{Z}}(y-x)$. It is easy to observe that essentially this is nothing but the index we defined in Section 2.
\end{example}

Analogous to the definition of the writhe polynomial, for any $\mathfrak{g}\neq \sum1$ we use $a_{\mathfrak{g}}(K)$ to denote the sum of the writhes of all crossings which have index $\mathfrak{g}\in\mathds{Z}\mathfrak{G}_{BQ}$. For $\sum1$ we define $a_{\sum1}(K)$ to be the sum of the writhes of all crossings which have index $\sum 1$ minus $w(K)$. The next theorem follows directly from our constructions above, which can be regarded as an extension of Theorem 3.1.
\begin{theorem}
Let $L$ be a virtual link diagram, then for any finite biquandle $BQ$ and any $\mathfrak{g}\in\mathds{Z}\mathfrak{G}_{BQ}$, $a_{\mathfrak{g}}(L)$ is a virtual link invariant.
\end{theorem}

\begin{example}
As we mentioned in the beginning of this section, now we can define the chord indices for the real crossing points in a virtual link diagram. Let us consider the virtual link $L$ in Figure \ref{figure20}. Choose a biquandle $X=\{1, 2\}$, and the binary operations are defined as $1\ast i=1\circ i=2$ and $2\ast i=2\circ i=1$ $(i=1, 2)$. It is easy to observe that in $\mathfrak{G}_X$ we have $1=(1, 1)=(2, 2), (1, 2)=(2, 1)$. Therefore $\mathfrak{G}_X\cong\mathds{Z}$, which is generated by $t=(1, 2)$. Note that there exist four colorings and the indices of crossing points $a, b, c$ are $1+1+t+t, t+t+t+t, 1+1+t+t$ respectively. As a result we have $a_{1+1+t+t}(L)=2$ and $a_{t+t+t+t}(L)=1$. As a corollary, we conclude that the real crossing number of this virtual link is 3.
\begin{figure}[h]
\centering
\includegraphics{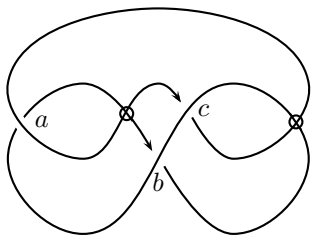}\\
\caption{Chord indices of a virtual link}\label{figure20}
\end{figure}
\end{example}

Since in general the biquandle structure is more complicated than the quandle structure, a natural thought is to replace the biquandles in our construction with quandles. Obviously, because a quandle is also a biquandle, we can still define the chord index in this case. However the following proposition tells us it provides no new information except the coloring invariant.
\begin{proposition}
Let $Q$ be a finite quandle and $K$ a virtual knot diagram, then all the crossing points of $K$ have the same index $\sum\limits_{Col_Q(K)}1$.
\end{proposition}
\begin{proof}
Fix a coloring $f$, if $Q$ is not connected, then let us focus on the component $Q'$  which includes $f(Q_K)$ (recall that $Q_K$ is connected). We claim that the group $\mathfrak{G}_{Q'}$ is trivial. In fact, according to the definition of $\mathfrak{G}_{Q'}$ we have relation $(x, z)=(x\ast y, z)$. Since $Q'$ is connected, there exist a sequence of elements in $Q'$, say $\{a_1, \cdots, a_n\}$, such that $(\cdots(x\ast^{\epsilon_1} a_1)\cdots)\ast^{\epsilon_n} a_n=z$ $(\epsilon_i=\pm1)$. Now we have
\begin{center}
$(x, z)=(x\ast^{\epsilon_1}a_1, z)=\cdots=((\cdots(x\ast^{\epsilon_1} a_1)\cdots)\ast^{\epsilon_n} a_n, z)=(z, z)=1$,
\end{center}
which means $\mathfrak{G}_{Q'}$ contains only one element. The result follows.
\end{proof}

In the end of this section we would like to remark that the chord index in Example 6.2 and Example 6.4 are both trivial for any crossing point in a classical knot diagram. It is natural to ask whether it is possible to define a nontrivial chord index for the crossing points in a classical knot diagram.

\section{Chord index in twisted knot theory}
\subsection{Twisted knot theory and twisted biquandle}
In the end of this paper we concern the chord index in twisted knot theory. Recall that virtual knot theory studies the embeddings of $S^1$ in thickened closed orientable surfaces, if we do not require that the surface must be orientable, then we encounter the twisted knot theory. Twisted knot theory was first proposed by Bourgoin in \cite{Bou2008}. A \emph{twisted knot} is a stable equivalence class of $S^1$ in oriented 3-manifolds that are $I$-bundles over closed but not necessarily orientable surfaces. One main result in \cite{Bou2008} generalizes Kuperberg's result \cite{Kup2003} from orientable surfaces to nonorientable surfaces. More precisely, Bourgoin mimicked Kuperberg's approach to prove that the irreducible representative of a twisted knot is unique. Therefore twisted knot theory is a proper extension of virtual knot theory.

One can also use twisted knot diagrams to illustrate twisted knots. A \emph{twisted knot diagram} is a virtual knot diagram with some bars on edges. We say two twisted knot diagrams are \emph{equivalent} if they are related by a sequence of generalized Reidemeister moves (see Figure \ref{figure1}) and twisted Reidemeister moves (see Figure \ref{figure16}). Similar to virtual knots, from each twisted knot diagram one can obtain an embedding of $S^1$ in a thickened surface, where each bar corresponds to a half-twist. It was proved in \cite{Bou2008} that these two definitions of twisted knots are equivalent.
\begin{figure}[h]
\centering
\includegraphics{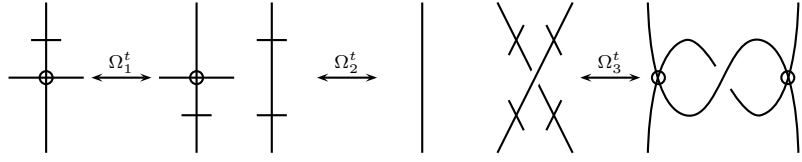}\\
\caption{Twisted Reidemeister moves}\label{figure16}
\end{figure}

The twisted knot group and the twisted Jones polynomial was defined in \cite{Bou2008}. Later Naoko Kamada generalized the arrow/Miyazawa polynomial to twisted knots \cite{Kam2012}. Using the similar idea of last section we want to introduce the chord index of twisted knots with a twisted biquandle. Note that the twisted quandle has been introduced by Naoko Kamada in \cite{Kam2012}, which was motivated by the twisted group defined by Bourgoin. Here the twisted biquandle discussed below can not be thought of as a biquandle version of the twisted quandle simply, although they reduce to biquandle and quandle respectively when the knot diagram contains no bar.
\begin{definition}
A \emph{twisted biquandle} is a biquandle $(BQ, \ast, \circ)$ with an additional map $f: BQ\rightarrow BQ$ which satisfies the following axioms
\begin{enumerate}
  \item $f(y\circ x)\ast f(x\ast y)=f(y)$,
  \item $f(x\ast y)\circ f(y\circ x)=f(x)$,
  \item $f^2(x)=x$.
\end{enumerate}
\end{definition}

With a given finite twisted biquandle $(BQ, \ast, \circ, f)$ we can define a coloring invariant $Col_{BQ}(K)$ for each twisted knot $K$ as follows. Choose a twisted knot diagram of $K$, for simplicity we still use $K$ to denote it. Assume there are $c_r(K)$ real crossing points and $b$ bars in $K$, now these crossings and bars split $K$ into $2c_r(K)+b$ segments. For each segment we label an element of $BQ$ to it such that the coloring rules described in Figure \ref{figure15} are satisfied. In additional, if two segments are adjacent to the same bar then the elements on them differ by $f$.
\begin{proposition}
$Col_{BQ}(K)$ is a twisted knot invariant.
\end{proposition}
\begin{proof}
The invariances of $Col_{BQ}(K)$ under generalized Reidemeister moves are guaranteed by the axioms of biquandle (see Definition 5.2). Hence it is sufficient to check the twisted Reidemeister moves in Figure \ref{figure16}. For $\Omega_1^t$, there is nothing need to prove. For $\Omega_2^t$, the invariance of $Col_{BQ}(K)$ follows from the fact that $f$ is an involution. Figure \ref{figure17} explains why $Col_{BQ}(K)$ is invariant under $\Omega_3^t$.
\end{proof}
\begin{figure}[h]
\centering
\includegraphics{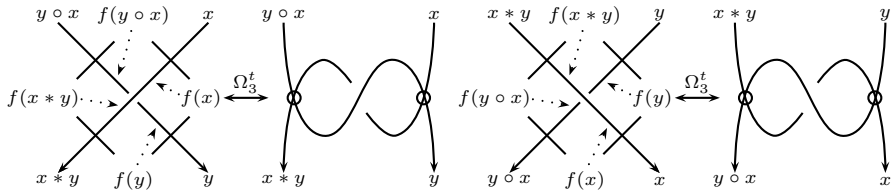}\\
\caption{The invariance of $Col_{BQ}(K)$ under $\Omega_3^t$}\label{figure17}
\end{figure}

In next subsection we will focus on a special twisted biquandle $(\mathds{Z}, a\ast b=a\circ b=a+1, f(a)=-a)$. We will show how to associate an index for each real crossing point via this twisted biquandle.

\subsection{A concrete example of chord index for twisted knots}
We consider the colorings of twisted knots using twisted biquandle $(\mathds{Z}, a\ast b=a\circ b=a+1, f(a)=-a)$. A naive observation is the parity of the number of bars is preserved under twisted Reidemeister moves. For example if a twisted knot diagram contains an odd number of bars then it can not represent a virtual knot. We continue our discussion in two cases.

First let we consider a twisted knot $K$ which has an odd number of bars. In this case, we have the following result.
\begin{lemma}
Let $K$ be a twisted knot diagram, if there are an odd number of bars in $K$, then the coloring is unique.
\end{lemma}
\begin{proof}
Assume $K$ has $2n-1$ bars, denoted by $b_1, \cdots, b_{2n-1}$. We use $e_1, \cdots, e_{2n-1}$ to denote the edges of $K-\{b_1, \cdots, b_{2n-1}\}$, where the order of $e_1, \cdots, e_{2n-1}$ is consistent with the orientation of $K$. Consider an edge $e_i$, we use $o_+(e_i), o_-(e_i), u_+(e_i), u_-(e_i)$ to denote the number of positive overcrossings, negative overcrossings, positive undercrossings and negative undercrossings on $e_i$ respectively. Then we assign an integer $s(e_i)=u_+(e_i)+o_-(e_i)-u_-(e_i)-o_+(e_i)$ to $e_i$.

Recall that by a coloring we mean an assignment of integers to the segments which are obtained from $K$ by deleting all real crossing points and bars. Along the direction of $K$, denote the segment adjoint to $b_1$ by $a$. Notice that according to the coloring rules, when we assign an integer $k$ to $a$ then the coloring of any other segment can be derived from the coloring on $a$. If a coloring is well-defined, then the derived coloring on $a$ must equal $k$. This can be described by the following equation
\begin{center}
$\sum\limits_{i=1}^{n-1}s(e_{2i})-\sum\limits_{i=1}^{n}s(e_{2i-1})-k=k$.
\end{center}
Since
\begin{center}
$\sum\limits_{i=1}^{n-1}s(e_{2i})+\sum\limits_{i=1}^{n}s(e_{2i-1})=\sum\limits_{i=1}^{2n-1}s(e_i)=0$.
\end{center}
It follows that $\sum\limits_{i=1}^{n-1}s(e_{2i})-\sum\limits_{i=1}^{n}s(e_{2i-1})$ is even, which means that $k$ has a unique solution.
\end{proof}

Now we know that for any twisted knot $K$ there is a unique coloring by using the twisted biquandle $(\mathds{Z}, a\ast b=a\circ b=a+1, f(a)=-a)$. If the colors on the four segments around a crossing $c$ are depicted as that in Figure \ref{figure15}, then we define the \emph{index} of $c$ to be Ind$(c)=y-x$. As before we define
\begin{center}
$a_n(K)=
\begin{cases}
\sum\limits_{\text{Ind}(c)=n}w(c)& \text{if }n\neq0;\\
\sum\limits_{\text{Ind}(c)=n}w(c)-w(K)& \text{if }n=0.
\end{cases}$
\end{center}
The following theorem is an analogue of Theorem 3.1 for twisted knots with an odd number of bars.
\begin{theorem}
Let $K$ be a twisted knot with an odd number of bars, then each $a_n(K)$ is a twisted knot invariant. Or equivalently, the polynomial $T_o(K)=\sum\limits_{n\in\mathds{Z}}a_n(K)t^n$ is a twisted knot invariant.
\end{theorem}
\begin{proof}
The invariance of $a_n(K)$ under the generalized Reidemeister moves follows directly from Proposition 2.3. For $\Omega_1^t$ and $\Omega_2^t$ there is nothing need to prove. For $\Omega_3^t$, the proof can be read from Figure \ref{figure17} by taking $a\ast b=a\circ b=a+1$ and $f(a)=-a$.
\end{proof}

Now we turn to the twisted knots with an even number of bars. Note that this set contains all virtual knots, and hence all classical knots. Let $K$ be a twisted knot with $2n$ bars, denoted by $b_1, \cdots, b_{2n}$. These bars divides the knot diagram $K$ into $2n$ edges, say $e_1, \cdots, e_{2n}$, where the order of $e_1, \cdots, e_{2n}$ agrees with the direction of $K$. For each edge $e_i$ the assigned integer $s(e_i)$ can be defined as above.
\begin{lemma}
$S(K)=|\sum\limits_{i=1}^ns(e_{2i-1})-\sum\limits_{i=1}^ns(e_{2i})|$ is invariant under the generalized Reidemeister moves and twisted Reidemeister moves.
\end{lemma}
\begin{proof}
First note that $S(K)$ does not depend on the choice of the first edge $e_1$, hence it is well-defined. Let us consider the generalized Reidemeister moves and twisted Reidemeister moves individually.
\begin{itemize}
  \item $\Omega_1$: assume $\Omega_1$ is taken on $e_i$, then $e_i$ adds a new overcrossing and a new undercrossing which have the same writhe. Therefore $s(e_i)$ is preserved.
  \item $\Omega_2$: in this case, there are two edges, say $e_i, e_j$, where $e_i$ adds a pair of new overcrossings with opposite signs and $e_j$ adds a pair of new undercrossings with opposite signs. Hence both $s(e_i)$ and $s(e_j)$ are invariant.
  \item $\Omega_3$: it is easy to observe that for each edge involved in $\Omega_3$ only the positions of two crossing points are switched.
  \item $\Omega_1', \Omega_2', \Omega_3', \Omega_3^v, \Omega_1^t$: nothing need to prove in these cases.
  \item $\Omega_2^t$: we use $e_i^l$ $(1\leq i\leq 2n+2)$ and $e_j^r$ $(1\leq j\leq 2n)$ to denote the edges on the left side and right side of $\Omega_2^t$ respectively (see Figure \ref{figure16}). Without loss of generality, assume the edge on the right side of $\Omega_2^t$ is $e_{2n}^r$. Then the three edges on the left side of $\Omega_2^t$ are $e_{2n}^l, e_{2n+1}^l, e_{2n+2}^l$. Note that $s(e_{2n}^l)+s(e_{2n+2}^l)=s(e_{2n}^r)$ and $s(e_{2n+1}^l)=0$. The result follows.
  \item $\Omega_3^t$: let us consider the two local diagrams on the left side of Figure \ref{figure17}, say $K$ and $K'$. We use $e_i$ and $e_i'$ to denote the edges of $K$ and $K'$ respectively. Let $e_i'$ be the edge that contains the curve that begins with $y\circ x$ and ends with $y$. It corresponds to three edges in $K$, say $e_i, e_{i+1}, e_{i+2}$ (since $S$ is well-defined, the first edge $e_1$ can be chosen away from the local diagram depicted in Figure \ref{figure17}). It is easy to observe that $s(e_i)+s(e_{i+2})=s(e_i')+1$ and $s(e_{i+1})=1$, hence we have $s(e_i)-s(e_{i+1})+s(e_{i+2})=s(e_i')$. The other cases in Figure \ref{figure17} can be checked in the same way.
\end{itemize}
\end{proof}

According to the definition of $S(K)$, it is evident that $S(K)=0$ if $K$ is a virtual knot. On the other hand since $\sum\limits_{i=1}^ns(e_{2i-1})+\sum\limits_{i=1}^ns(e_{2i})=0$, we conclude that $S(K)$ is always an even integer. The relation between the existence of colorings of $K$ and $S(K)$ is given in the following lemma.
\begin{lemma}
Consider the twisted biquandle $(\mathds{Z}, a\ast b=a\circ b=a+1, f(a)=-a)$, there exists a coloring of $K$ if and only if $S(K)=0$.
\end{lemma}
\begin{proof}
The proof is similar to the proof of Lemma 7.3. If an segment is assigned with an integer $k$, then the labels on other segments are determined according to the coloring rules. It is easy to find that there exists a coloring if and only if the following equations
\begin{center}
$\sum\limits_{i=1}^{n}s(e_{2i-1})-\sum\limits_{i=1}^{n}s(e_{2i})+k=k$ and $\sum\limits_{i=1}^{n}s(e_{2i})-\sum\limits_{i=1}^{n}s(e_{2i-1})+k=k$
\end{center}
hold. Consequently, there exists a coloring of $K$ if and only if $S(K)=0$.
\end{proof}

According to the proof above, we know that if $S(K)=0$ then there are infinitely many different colorings, which can be obtained by coloring a fixed segment with all integers. In order to color twisted knots with nonzero $S(K)$, we replace the twisted biquandle $(\mathds{Z}, a\ast b=a\circ b=a+1, f(a)=-a)$ with $(\mathds{Z}_{S(K)}, a\ast b=a\circ b=a+1, f(a)=-a)$. In particular, if $S(K)=0$ then we have $\mathds{Z}_{S(K)}=\mathds{Z}$. For a twisted knot $K$, if we use the twisted biquandle $(\mathds{Z}_{S(K)}, a\ast b=a\circ b=a+1, f(a)=-a)$ then there are exactly $S(K)$ different colorings. Fix a coloring $f$, for the crossing point $c$ depicted in Figure \ref{figure15} we define the index of it associated to $f$ as Ind$_f(c)=y-x ($mod $S(K)) \in \mathds{Z}_{S(K)}$. Obviously this definition depends on the choice of $f$, but what we need is an index which does not depend on the choice of colorings. Hence it is necessary to study the indices of all colorings. Fortunately, there are only $S(K)$ different colorings totally.

Let $K$ be a twisted knot with an even number of bars, the Gauss diagram of $K$ can be similarly defined as the virtual knots. We still use $G(K)$ to denote it. For any chord $c$ in $G(K)$, it splits the circle into two semi-circles. Since there are totally an even number of bars in $K$, then either both semi-circles have an even number of bars, or both semi-circles have an odd number of bars. We use $C_e(K)$ to denote all the chords of the first case and $C_o(K)$ to denote all the chords of the second case.
\begin{lemma}
Choose $c_1\in C_e(K)$ and $c_2\in C_o(K)$, then Ind$_f(c_1)$ does not depend on the choice of $f$, and Ind$_f(c_2)$ take values on all odd or all even numbers of $\mathds{Z}_{S(K)}$.
\end{lemma}
\begin{proof}
Recall that all colorings of $K$ can be obtained by coloring a fixed segment $a$ with all integers in $\mathds{Z}_{S(K)}$. Assume that when we assign $0$ to $a$, the index of $c_1$ equals $y_1-x_1$ and the index of $c_2$ equals $y_2-x_2$ (see Figure \ref{figure15}). Then if $a$ is colored by some integer $k$, the index of $c_1$ turns into $(y_1\pm k)-(x_1\pm k)=y_1-x_1$. However, the index of $c_2$ becomes $(y_2\pm k)-(x_2\mp k)=y_2-x_2\pm2k$. The proof is finished.
\end{proof}

Due to the lemma above the set $C_o(K)$ can be divided into two pieces, say $C_o^0(K)$ and $C_o^1(K)$, where $C_o^0(K)$ contains all the crossing points with even indices and $C_o^1(K)$ contains all the crossing points with odd indices. Since Ind$_f(c)$ does not depend on the choice of $f$ if $c\in C_e(K)$, we can simply use Ind$(c)$ to denote it. The results of the discussion above can be summarized in the form of a polynomial
\begin{center}
$T_e(K)=\sum\limits_{c\in C_o^0(K)}w(c)s_0+\sum\limits_{c\in C_o^1(K)}w(c)s_1+\sum\limits_{c\in C_e(K)}w(c)t^{\text{Ind}(c)}-w(K)$.
\end{center}
\begin{theorem}
Let $K$ be a twisted knot with an even number of bars, then $T_e(K)$ is a twisted knot invariant.
\end{theorem}
It is routine to check that $\sum\limits_{c\in C_o^0(K)}w(c), \sum\limits_{c\in C_o^1(K)}w(c)$ and $\sum\limits_{c\in C_e(K)}w(c)t^{\text{Ind}(c)}-w(K)$ are invariant under the generalized Reidemeister moves and the twisted Reidemeister moves. Therefore we omit the proof here.

The first index type invariant of twisted knots was introduced by Naoko Kamada in \cite{Kam2013}. In \cite{Kam2013}, Naoko Kamada defined two polynomials of twisted knots, denoted by $\overline{Q}_K$ and $\widetilde{Q}_K$, where $\widetilde{Q}_K$ is a refinement of $\overline{Q}_K$. We remak that some of our results, for instance $\sum\limits_{c\in C_o^0(K)}w(c)s_0$ and $\sum\limits_{c\in C_o^1(K)}w(c)s_1$ also can be found in the definition of $\overline{Q}_K$. We end this paper with an example which explains the difference between our polynomial invariants and that introduced by Naoko Kamada.
\begin{example}
Consider the twisted knot $K$ described in Figure \ref{figure18}. It has three bars hence there is a unique coloring. Direct calculation shows that $T_o(K)=2t^2+t^{-4}-3$. However we remark that the chord index used in $\overline{Q}_K$ and $\widetilde{Q}_K$ \cite{Kam2013} was defined in a similar manner as the linking number definition we mentioned in Section 2. For the twisted knot in Figure \ref{figure18}, notice that smoothing any crossing point will give us a 2-component split link. It follows that $\overline{Q}_K$ and $\widetilde{Q}_K$ are both trivial in this example.
\end{example}
\begin{figure}[h]
\centering
\includegraphics{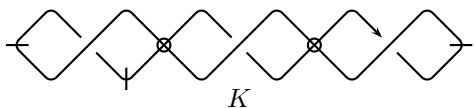}\\
\caption{A twisted knot with an odd number of bars}\label{figure18}
\end{figure}

\section*{Acknowledgement}
This paper is completed during the author's visit to the George Washington University. The author appreciates their kind hospitality during his visit. The author is supported by NSFC 11301028, NSFC 11571038 and China Scholarship Council.

\bibliographystyle{amsplain}

\end{document}